\documentclass[11pt]{amsart}
\usepackage{amsfonts}							
\usepackage{amsmath}							
\usepackage{amssymb}							
\usepackage{amsthm}								
\usepackage{dsfont}									
\usepackage[dvipsnames]{xcolor}				
\usepackage[]{hyperref} 						
\usepackage{mathtools}								
\usepackage{tikz} 										
\usepackage{tikz-cd}									
\DeclareMathAlphabet{\mathpzc}{OT1}{pzc}{m}{it}
\usepackage{wrapfig}								

\makeatletter
\def\@seccntformat#1{%
	\protect\textup{\protect\@secnumfont
		\ifnum\pdfstrcmp{subsection}{#1}=0 \bfseries\fi
		\csname the#1\endcsname
		\protect\@secnumpunct\hspace{-.65em}
	}%
}  
\makeatother

\usepackage[left=1in,top=1in,right=1in,bottom=1in,head=.2in]{geometry}

\makeatletter
\g@addto@macro \normalsize {
	\setlength\abovedisplayskip{2pt plus 0pt minus 2pt}
	\setlength\belowdisplayskip{2pt plus 0pt minus 2pt}
}
\makeatother

\newtheorem{theorem}{Theorem}[section] 

\newtheorem{proposition}[theorem]{Proposition}
\newtheorem{corollary}[theorem]{Corollary}
\newtheorem*{theorem*}{Theorem}
\newtheorem*{lemma*}{Lemma}
\newtheorem*{proposition*}{Proposition}
\newtheorem*{corollary*}{Corollary}

\theoremstyle{definition}

\newtheorem{remark}[theorem]{Remark}

\newtheorem*{definition*}{Definition}
\newtheorem*{remark*}{Remark}
\newtheorem*{remarks*}{Remarks}
\newtheorem*{example*}{Example}
\newtheorem*{examples*}{Examples}

\renewenvironment{proof}{\noindent \textbf{Proof}.}{\qed \vskip5pt}

\renewcommand\emptyset{\varnothing}
\renewcommand\geq{\geqslant}
\renewcommand\leq{\leqslant}

\renewcommand\:{\colon}
\renewcommand\bar[1]{\overline{#1}}

\newcommand{\R}{\mathds{R}}
\newcommand{\Z}{\mathds{Z}}

\newcommand{\ckh}{\mathcal{C}\textnormal{Kh}}
\newcommand{\ckj}{\mathcal{C}\textnormal{KJ}}
\newcommand{\kh}{\textnormal{Kh}}
\newcommand{\kj}{\textnormal{KJ}}

\newcommand{\crossing}{\raisebox{-.2\height}{\includegraphics[scale=.2]{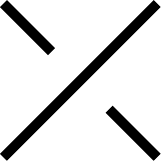}} \null}
\newcommand{\zsmooth}{\raisebox{-.2\height}{\includegraphics[scale=.2]{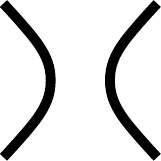}} \null}
\newcommand{\osmooth}{\raisebox{-.2\height}{\includegraphics[scale=.2]{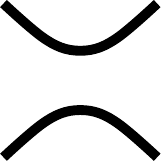}} \null}

\title{Relative Khovanov-Jacobsson Classes}
\author{Isaac Sundberg \& Jonah Swann}

\usepackage{fancyhdr}
\usepackage{mathrsfs}
\begin{document}

%
%

\begin{abstract}
	To a smooth, compact, oriented, properly-embedded surface in the $4$-ball, we define an invariant of its boundary-preserving isotopy class from the Khovanov homology of its boundary link. Previous work showed that when the boundary link is empty, this invariant is determined by the genus of the surface. We show that this relative invariant: can obstruct sliceness of knots; detects a pair of slices for $9_{46}$; is not hindered by detecting connected sums with knotted $2$-spheres.
\end{abstract}

\maketitle

\setcounter{tocdepth}{3}

%
%

\begin{section}{{Introduction}}
	
	Conjectured within the framework of \cite{khovanov00}, which defined what would become known as Khovanov homology, was an application of the map on Khovanov homology induced by an oriented link cobordism. More precisely, to an oriented link cobordisms $\Sigma\: L_0 \to L_1$, Khovanov constructed a diagrammatically-defined homomorphism $\kh(\Sigma)\: \kh(L_0) \to \kh(L_1)$ and conjectured that, up to sign, it is invariant under isotopy of $\Sigma$ rel $\partial \Sigma$. Specifically, Khovanov noted that if this conjecture held, then the following application might warrant further study: for a closed surface $S\: \emptyset \to \emptyset$, the induced map is an endomorphism $\kh(S)\: \Z \to \Z$, whereby the integer $n_S := |\kh(S)(1)| \in \Z$ determining $\kh(S)$ is an invariant of the ambient isotopy class of $S$. The conjecture was eventually proven by Jacobsson \cite{jacobsson04}, and then later in \cite{barnatan05} and \cite{khovanov06} through independent methods. The integer $n_S$ became known as the \textit{Khovanov-Jacobsson number} of the surface $S$, and calculations of $n_S$ for certain families of surface knots were established \cite{cartersaitosatoh04}. In the end, Rasmussen \cite{rasmussen05} and Tanaka \cite{tanaka05} independently proved Khovanov-Jacobsson numbers are trivially determined by the genus of $S$.
	
	In this paper, we define a relative version of Khovanov-Jacobsson numbers associated to link cobordisms $\Sigma\:\emptyset \to L$. Namely, the \textit{relative Khovanov-Jacobsson class} of $\Sigma$ is the class
		$$\kj_\Sigma = |\kh(\Sigma)(1)| \in \kh(L).$$
	It is an invariant of the isotopy class of $\Sigma$ rel $\partial \Sigma$. When $L = \emptyset$, we use \textit{absolute} in place of relative. A definition of this invariant, and a discussion of its diagram dependence, is given in Section \ref{rkjc} \!.
	
	The remainder of the paper studies and implements properties of relative Khovanov-Jacobsson classes. In Section \ref{crkjc} \!, we characterize the relative classes associated to surfaces bounding the unlink and surfaces obtained by Seifert's algorithm. In Section \ref{osk} \!, we show these classes can obstruct sliceness of knots. We illustrate this on a class of $3$-stranded pretzel knots.
	
	The paper concludes by discussing boundary-preserving isotopy classes of surfaces properly-embedded in the $4$-ball. Such classes have gained recent attention in \cite{juhaszzemke20}, \cite{conwaypowell20}, and \cite{millerpowell20}, where a refined characterization has been suggested. Connect summing a surface with a knotted $2$-sphere, or \textit{locally knotting} a surface, generally changes the boundary-preserving isotopy class of the surface. By considering surfaces up to both boundary-preserving isotopy and local knotting, we create a more nuanced distinction between surfaces, better capturing the boundary-preserving isotopy being established. In Section \ref{obpic} \!, we show that relative Khovanov-Jacobsson classes do not detect local knottedness, implying they are an invariant of these refined isotopy classes of surfaces. We conclude by showing relative Khvanov-Jacobsson classes detect a pair of slices for $9_{46}$ and generalize this to arbitrary connect sums of $9_{46}$.
	\vspace{10pt}
	
\end{section}

%
%

\begin{section}{{Background on Khovanov Homology}} \label{bkh}
	
	In this section, we provide the notation and jargon we use when discussing the Khovanov functor. We assume familiarity with the construction of the Khovanov chain complex, see \cite{khovanov00} or \cite{barnatan02} among many others. Sections 2.2 is based on the terminology from \cite{elliott09_2}, and Section 2.3 is based on the work from \cite{jacobsson04}.
	
	\subsection{{Khovanov Homology}} \label{bkh_homology}
	To each diagram $D$ of an oriented link $L$ and each enumeration of its crossings, we associate a bigraded chain complex $\ckh(D)$ called the \textit{Khovanov chain complex} of $D$. The homology $\kh(D)$ is called the \textit{Khovanov homology} of $D$. This definition is independent, up to chain homotopy equivalence, of both the chosen enumeration of crossings in $D$ and the chosen diagram $D$. In particular, associated to each Reidemeister move is a pseudoisomorphism between Khovanov chain complexes, discussed briefly in Section \ref{bkh_maps} \!. Thus, the chain homotopy class $\ckh(L)$ of $\ckh(D)$ is an invariant of the link $L$ called the \textit{Khovanov chain complex} of $L$; the isomorphism class $\kh(L)$ of $\kh(D)$ is called the \textit{Khovanov homology} of $L$.
	
	\subsection{{Elements of the Khovanov chain complex}} \label{bkh_elements}
	Here, we provide language used on elements of the Khovanov chain complex. Let $D$ be a diagram for an oriented link $L$ having $n$ crossings, of which there are $n_+$ positive crossings and $n_-$ negative crossings. Enumerate the crossings of $D$.
	
	A crossing $\crossing$ of $D$ can be \textit{smoothed} as either a $0$-\textit{smoothing} $\zsmooth$ or a $1$-\textit{smoothing} $\osmooth$. 
	
	A \textit{state} $\sigma$ of $D$ is a planar $1$-manifold obtained by replacing each crossing in $D$ with either a $0$-smoothing or $1$-smoothing. A diagram with $n$-many crossings will have $2^n$-many states. A state can be represented as a binary sequence $\sigma = (v_1, \dots, v_n)$, where the $i$th crossing is $v_i$-smoothed. The \textit{height} of a smoothing is $|\sigma| = \sum v_i$. The state whose crossings are all $0$-smoothed (or $1$-smoothed) is called the $0$-\textit{state} (or $1$-\textit{state}).
	
	A \textit{trace} is an indicator placed on a smoothed crossing to record the type of smoothing: a $0$-\textit{trace} is a red dot-dashed band indicating a $0$-smoothing; a $1$-\textit{trace} is a blue dotted band indicating a $1$-smoothing. A \textit{trace state} is a state whose smoothings are labeled with the appropriate traces.
	
	An \textit{enhanced state} $\alpha_\sigma$ is a state $\sigma$ whose components are labeled with either a $1$ or an $x$. A state with $m$ components has $2^m$ associated enhanced states. We record the number $v_+(\alpha)$ of $1$-labels and the number $v_-(\alpha)$ of $x$-labels. The enhanced state whose loops are all $1$-labeled is called the \textit{all} $1$-\textit{label} and is denoted $\sigma_1$; the enhanced state whose loops are all $x$-labeled is called the \textit{all} $x$-\textit{label} and is denoted $\sigma_x$. An \textit{enhanced trace state} is both an enhanced state and a trace state. 
	
	We define a bigrading $\ckh^{h,q}(D)$ on enhanced states having \textit{homological} and \textit{quantum} gradings:
	\begin{align*}
		h(\alpha_\sigma) &= |\sigma| - n_- \\
		q(\alpha_\sigma) &= v_+(\alpha_\sigma) - v_-(\alpha_\sigma) + h(\alpha_\sigma) + n_+ - n_-
	\end{align*}
	Each chain group $\ckh^{h,q}(D)$ is generated by the possible enhanced states from this bigrading\!
		\footnote{Formally, the generators of $\ckh(D)$ are algebraic objects which can be viewed pictorially as enhanced states.}
	\!.
	
	\begin{figure}[!ht]
		\includegraphics[scale=.3]{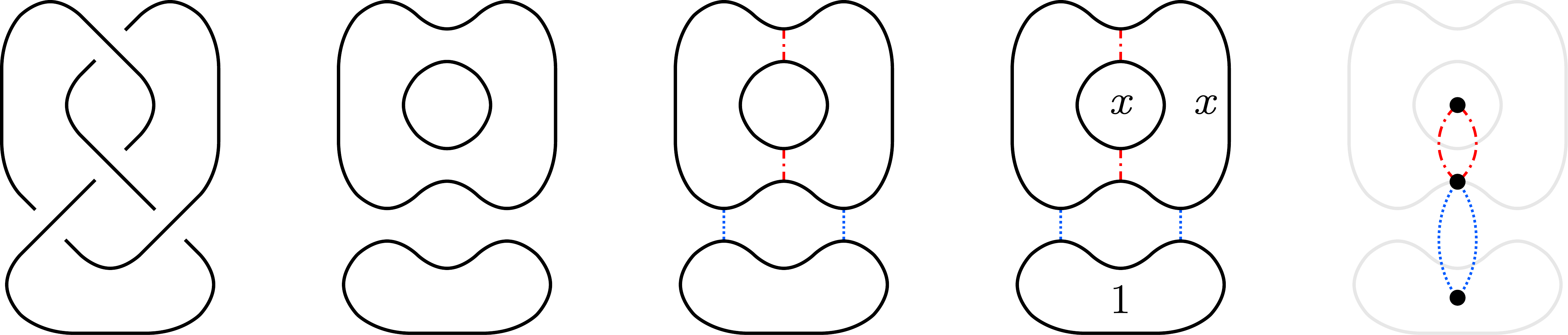}
		\caption{From left to right: the figure $8$ knot; a state $\sigma$ of the figure $8$ knot; the trace state of $\sigma$; one of the $2^3$ enhanced trace states on $\sigma$; the trace state graph $\Gamma_\sigma$.}
		\label{fig_trace}
	\end{figure}
	
	Associated to a trace state $\sigma$ is a \textit{trace state graph} $\Gamma_\sigma$ whose vertices are the components of $\sigma$ and whose edges are the traces. When illustrating $\Gamma_\sigma$, edges can be drawn using bands that mirror the trace they represent. Moreover, we often show $\sigma$ as a shadow of the graph, as in Figure \ref{fig_trace} \!. We interchangeably refer to the loops (or traces) in $\sigma$ and the vertices (or edges) in $\Gamma_\sigma$. We will have occasion to refer to certain subgraphs of the state graph.	Let $\Gamma_0$ denote the subgraph whose edges consist of all $0$-traces, and let  $\Gamma_1$ be the subgraph of $1$-traces.
	
	One chain element that we encounter regularly in this paper is called a $pqr$-chain, which we have illustrated in Figure \ref{fig_pqr} \!. The left side of the figure shows the shorthand we use to describe such chains. In this shorthand, a loop $\gamma$ labeled with a letter of the alphabet (usually beginning with $p$ and proceeding alphabetically) corresponds to a summand in the $pqr$-chain where $\gamma$ is $1$-labeled and all other labeled loops are $x$-labeled. If the letter on $\gamma$ is capitalized, the summand has a negative sign. A loop that is unlabeled in the shorthand is $1$-labeled in each summand. Using $pqr$-chains significantly reduces the number of enhanced states that must be illustrated in certain chain elements. We will make particularly good use of this in the upcoming sections.
	
	\begin{figure}[!ht]
		\includegraphics[scale=.3]{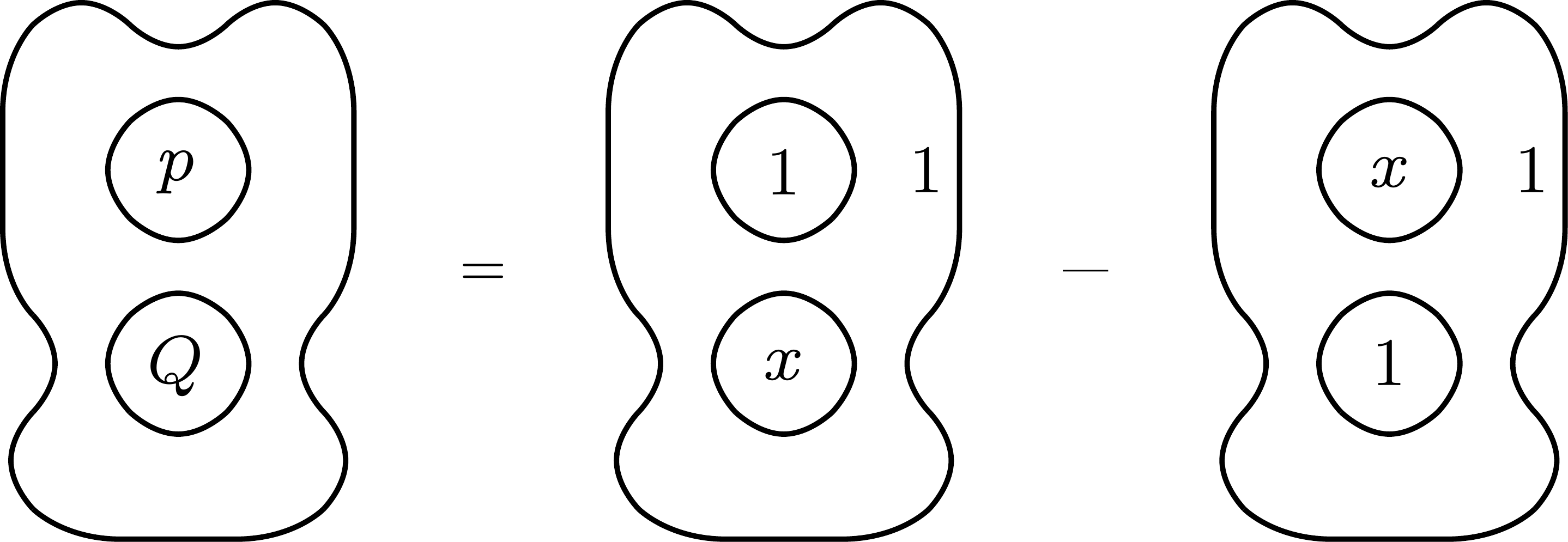}\vspace{-10pt}
		\caption{A $pqr$-chain on the all $0$-state of the figure $8$ knot .}
		\label{fig_pqr}
	\end{figure}
	\vspace{-5pt}
	\subsection{{Maps on Khovanov homology}} \label{bkh_maps}
	Here, we describe the map on Khovanov homology induced by a link cobordism.	A link cobordism $\Sigma\: L_0 \to L_1$ is a smooth, compact, oriented, properly-embedded surface $\Sigma \subset \R^3 \times [0,1]$ whose boundary is a pair of oriented links $L_i = \Sigma \cap (\R^3 \times \{i\})$. A generic
		\footnote{A link cobordism $\Sigma\: L_0 \to L_1$ is \textit{generic} if, with respect to the $[0,1]$ factor of $\R^3 \times [0,1]$, it restricts to a Morse function with distinct critical values. When generic, each level set $L_t = \Sigma \cap (\R^3 \times \{t\})$ is a link except at finitely many \textit{critical levels} $t_1, \dots, t_n$, where it contains either a transverse double point or an isolated point. In this paper, we assume all link cobordisms are generic.}
	link cobordism with surface diagram	
		\footnote{A \textit{surface diagram} $S\: D_0 \to D_1$ of a generic link cobordism $\Sigma\: L_0 \to L_1$ is the image $S \subset \R^2 \times I$ of $\Sigma$ under a generic projection $(p \times \text{id})\: \R^3 \times I \to \R^2 \times I$. Under these conditions, each $D_i = S \cap (\R^2 \times \{i\})$ is a diagram of the boundary link $L_i$ except at finitely many \textit{critical levels} $t_1, \dots, t_n$, corresponding to the critical levels of $\Sigma$.} 
	$S\: D_0 \to D_1$ induces a movie\!
		\footnote{A \textit{movie} $D_t$ associated to a surface diagram $S\: D_0 \to D_1$ of a generic link cobordism $\Sigma\: L_0 \to L_1$ is a collection of \textit{frames} $D_t = S \cap \R^2 \times \{t\}$. In a sufficiently small neighborhood $[t_i-\varepsilon(i), t_i + \varepsilon(i)]$ of each critical level $t_i$ of $S$, the diagrams $D_{t_i - \varepsilon(i)}$ and $D_{t_i + \varepsilon(i)}$ are related by either a Morse move or Reidemeister move. In each interval $(t_i, t_{i+1})$, the movie describes an isotopy between $D_{t_i + \varepsilon(i)}$ and $D_{t_{i+1} - \varepsilon(i+1)}$. Thus, a movie can be realized as a finite sequence of isotopies, Morse moves, and Reidemeister moves applied to $D_0$ and ending with $D_1$.}
	$D_t$. This movie can be decomposed into a finite sequence of link diagrams, with each successive pair related by an isotopy, Morse move, or Reidemeister move. These finite sequences of diagrams provide a way to explicitly express a link cobordism.
	
	For diagrams $D_0$ and $D_1$ related by a Morse or Reidemeister move, there is a bigraded chain map $\ckh(D_0) \to \ckh(D_1)$. The Reidemeister induced maps are chain equivalences. Explicit statements of these chain maps exist throughout the literature; we will not give them here. The Morse induced maps are those used in constructing the Khovanov chain complex, and can be found in most resources presenting this construction. Explicit definitions of the Reidemeister moves are harder to come by. In \cite{elliott09_1}, chain maps for the first and second Reidemeister moves are defined on enhanced states. In \cite{barnatan05}, all but one Reidemeister III move is given using a version of Khovanov homology defined categorically with cobordisms.
	
	To define the maps on Khovanov homology induced by a link cobordism $\Sigma\: L_0 \to L_1$, we use the following process. Choose a surface diagram $S\: D_0 \to D_1$ for $\Sigma$ and express it as a movie $D_t$. Shorten the movie to a finite sequence of diagrams, with each successive pair of diagrams related by an isotopy, Morse move, or Reidemeister move. Each of these relations between pairs of diagrams corresponds to a chain map. Sequentially compose the chain maps according to the sequence of diagrams. The result is a bigraded chain map on the Khovanov chain complex
		$$\ckh(S)\: \ckh^{i,j}(D_0) \to \ckh^{i,j+\chi(\Sigma)}(D_1)$$
	inducing a bigraded homomorphism $\kh(S)$ on homology. This chain map is invariant, up to sign and chain homotopy, under isotopy of $\Sigma$ rel $\partial \Sigma$, as in the following theorem.
	
	\stepcounter{theorem}
	\begin{theorem} \label{bkh_theorem1}
		\cite{jacobsson04} If $S, S'\: D_0 \to D_1$ are surface diagrams for link cobordisms $\Sigma, \Sigma'\: L_0 \to L_1$ that are isotopic relative to their common boundary, then $\ckh(S) \sim \pm \ckh(S')$.
	\end{theorem}
	
	\noindent Thus, the induced homomorphism $\kh(S)$ is an up-to-sign invariant of the boundary-preserving isotopy class of $\Sigma$. Throughout this paper, we will use the common shorthand $\ckh(\Sigma) := \ckh(S)$ and $\kh(\Sigma) := \kh(S)$. The reader is warned that this can be misleading, as it does not specify the surface diagram with which the domain and codomain are defined. When one is not specified, assume we are working with a surface diagram, whereby the (co)domain of either map are defined by the associated link diagrams. 
	
	\begin{remark}
		The induced map is functorial with respect to composition\!
			\footnote{A pair of link cobordisms $\Sigma\: L_0 \to L_1$ and $\Sigma'\: -L_1 \to L_2$, can be composed by scaling $\Sigma$ into $\R^2 \times [0, \frac12]$ and $\Sigma'$ into $\R^2 \times [\frac12, 1]$, forming the link cobordism $\Sigma' \circ \Sigma\: L_0 \to L_2$.}
		of link cobordisms: for link cobordisms $\Sigma\: L_0 \to L_1$ and $\Sigma'\: -L_1 \to L_2$, the induced Khovanov chain maps satisfy
		$$\ckh(\Sigma') \circ \ckh(\Sigma) = \ckh(\Sigma' \circ \Sigma).$$
	Note that this composition makes sense, as $\ckh(L_1) = \ckh(-L_1)$. Conversely, a link cobordism can be decomposed as a composition of a pair of link cobordisms by cutting along some non-critical level of the cobordism. In such a case, the induced map also decomposes in the expected manner.
	\end{remark}
\end{section}

%
%

\begin{section}{{Relative Khovanov-Jacobsson Classes}} \label{rkjc}
	
	In this section, we define the relative Khovanov-Jacobsson cycle and class associated to a link cobordism $\emptyset \to L$. We discuss the dependence of each on the chosen diagram of the boundary link.
	
	\subsection{{Definition}} \label{rkjc_definition}
	In the previous section, we showed how a link cobordism $\Sigma\:\emptyset \to L$ with surface diagram $S\: \emptyset \to D$ induces a chain map $\ckh(S)\: \ckh(\emptyset) \to \ckh(D)$ with induced homomorphism $\kh(S)\: \kh(\emptyset) \to \kh(D)$. Since both $\ckh(\emptyset)$ and $\kh(\emptyset)$ are infinite cyclic, or more generally some commutative ring with unity $1$, this map is determined by the image of $1$. Theorem \ref{bkh_theorem1} then implies
		$$\kj_{\Sigma,D} := |\kh(\Sigma)(1)|$$
	is an invariant of the boundary-preserving isotopy class of $\Sigma$. We call this element the \textit{relative Khovanov-Jacobsson class} of $\Sigma$ with respect to the diagram $D$.
	
	In practice, the chain element $\ckh(\Sigma)(1)$ is more valuable than $\kj_{\Sigma,D}$, as it encodes all the necessary information for both maps $\ckh(\Sigma)$ and $\kh(\Sigma)$. For this reason, the element
		$$\ckj_{\Sigma,D} := \ckh(\Sigma)(1)$$ 
	is called the \textit{relative Khovanov-Jacobsson cycle}, and, up to sign, it represents $\kj_{\Sigma,D}$. 
	
	We should verify $\ckj_{\Sigma,D}$ is a cycle. The chain complex $\ckh(\emptyset)$ is supported in the $(0,0)$-grading, so we must have $d(1) = 0$. Because $\ckh_\Sigma$ is a chain map, it follows that
		$$d(\ckj_{\Sigma,D}) = d(\ckh(\Sigma)(1)) = \ckh(\Sigma)(d(1)) = \ckh(\Sigma)(0) = 0.$$
	So $\ckj_{\Sigma,D}$ is a cycle in $\ckh(D)$, representing a homology class in $\kh(D)$.
	
	As said in the introduction, the case where $L = \emptyset$ has been fully established. In this case, we call $\kj_{\Sigma,\emptyset}$ the \textit{Khovanov-Jacobsson number} of $\Sigma$. In light of this paper, one might call it an  \textit{absolute Khovanov-Jacobsson class}. Initial observation shows that, because $\ckh(\emptyset)$ is supported in the $(0,0)$-grading and $\ckh(\Sigma)$ is a $(0,\chi(\Sigma))$-graded map, we have $\kj_{\Sigma,\emptyset} = 0$ for surfaces with $g(\Sigma) \neq 1$. The case where $g(\Sigma) = 1$ is harder, but the Khovanov-Jacobsson numbers were established:

	\begin{theorem} \label{tanakarasmussen} \cite{rasmussen05} \cite{tanaka05} 
		A link cobordism $\Sigma\: \emptyset \to \emptyset$ with genus $g(\Sigma) = 1$ has absolute Khovanov-Jacobsson class $\kj_{\Sigma,\emptyset} = 2$.
	\end{theorem}

	\begin{remark}
		Relative Khovanov-Jacobsson classes bear some resemblance to the invariant from knot Floer homology defined in \cite{juhaszmarengon16}. It is unclear if the two invariants are related (for example, by the Dowlin spectral sequence).
	\end{remark}

	\subsection{{Diagram Dependence}} \label{rkjc_diagram}
	The relative Khovanov-Jacobsson class of a link cobordism $\Sigma\: \emptyset \to \!L$ depends on the chosen surface diagram $S\: \emptyset \to D$. More specifically, it depends on the chosen link diagram: a diagram for the boundary link induces a surface diagram for the link cobordism because both are defined by the same projection. This dependence is straightforward: the diagram specifies the chain group and homology group to which $\ckj_{\Sigma,D}$ and $\kj_{\Sigma,D}$ respectively belong. When we change the surface diagram, these groups change by a Reidemeister induced equivalence. To measure the diagram dependence of relative Khovanov-Jacobsson classes, one could ask whether they are preserved by the Reidemeister induced equivalences. In this way, one could say that relative Khovanov-Jacobsson classes are \textit{link cobordism invariants}, being invariant of the chosen surface diagram. We prove this fact in Proposition \ref{rkjc_proposition1}, however, we must first make a definition. 
	
	\begin{remark} \label{remark1}
		Because we are studying the boundary-\textit{preserving} isotopy class of $\Sigma$, we emphasize that this paper only considers the link $L$, and not its isotopy class. Moreover, we emphasize that the permissible chain homotopies defining $\ckh(L)$ arise from a sequence of \textit{link specific} Reidemeister moves (made precise below), and not by \textit{any} sequence of Reidemeister moves.
		
		A diagram for a link $L$ is the image of some projection $p\: \R^3 \to \R^2$ onto a codimension-one, linear subspace. Projections $p, p'\: \R^3 \to \R^2$ defining diagrams $D, D'$ are related by a one-parameter family of rotations $r_t\: \R^3 \times I \to \R^3$ with $r_0 = \text{id}$ taking one projection onto the other $p' = p \circ r_1$. A small perturbation of $r$ makes $p \circ r_t$ generic as a link projection, whereby $(p \circ r_t)(L)$ describes a sequence of \textit{link specific} Reidemeister moves from $D$ to $D'$, meaning each Reidemeister move relates a pair of diagrams specific to $L$. The chain map induced by these Reidemeister moves, and the map it induces on homology, will be called a \textit{link specific Reidemeister induced map}.
		
		Thus, $\ckh(L)$ denotes the chain homotopy class of $\ckh(D)$ up to link specific Reidemeister induced chain homotopy equivalences, and $\kh(L)$ denotes the isomorphism class of $\kh(D)$ under link specific Reidemeister induced isomorphisms.
	\end{remark}
	
	\begin{proposition} \label{rkjc_proposition1}
		For a link cobordism $\Sigma\: \emptyset \to L$ with surface diagrams $S, S'\: \emptyset \to D, D'$, this diagram commutes, up to sign and up to homotopy, for link specific Reidemeister induced maps $\varphi$.
	\end{proposition}
	\begin{center}
		\begin{tikzcd}[column sep = -5pt, row sep = 25pt]
			\ckh(D') \arrow[rr, "\varphi"] & & \ckh(D) \\
		 & \ckh(\emptyset) \arrow[ul, "\ckh(S')"] \arrow[ur, "\ckh(S)"']
		\end{tikzcd}
	\end{center}\vspace{-2.5pt}
	\textit{Thus, the class $\ckj_{\Sigma,L}$ of relative Khovanov-Jacobsson cycles $\ckj_{\Sigma,D}$ related by link specific Reidemeister induced chain maps, and the class $\kj_{\Sigma,L}$ defined similarly, are link cobordism invariants.} \\
	
	\begin{proof}
		The idea is to construct a pair of link cobordisms inducing $\varphi \circ \ckh(S')$ and $\ckh(S)$ that are isotopic relative to $L$, whereby Theorem \ref{bkh_theorem1} \! implies the diagram commutes in the desired manner. The tricky part is producing isotopic cobordisms whose boundaries produce identical diagrams \textit{with respect to the same projection}. 
		
		Let $p, p'\: \R^3 \to \R^2$ be the projections defining the surface diagrams $S$ and $S'$. Remark \ref{remark1} gives a one-parameter family of rotations $r_t$ with $r_0 = \text{id}$ and $p' = p \circ r_1$ that induces a movie $(p \circ r_t)(L)$ describing a sequence of link specific Reidemeister moves from $D$ to $D'$. Consider the one-parameter family of link cobordisms describing the trace of $L$ under this movie:
			$$A_s = r|_{L \times [0,1-s]}\: r_0(L) \to r_{1-s}(L)$$
		The link cobordism $A_0^{-1}\: r_1(L) \to L$ has surface diagram $p(A_0^{-1})\: D' \to D$, so with respect to $p$, this link cobordism $A_0^{-1} \circ r_1(\Sigma)$ induces the map $\varphi \circ \ckh(S')$. Since $\Sigma$ induces $\ckh(S)$ with respect to $p$, it suffices to show $A_0^{-1} \circ r_1(\Sigma)$ and $\Sigma$ are isotopic relative to $L$. Indeed, $A_s^{-1} \circ r_s(\Sigma)$ describes a boundary-preserving isotopy between these link cobordisms, as desired.
	\end{proof}
	
	In light of this proposition, the diagram defining a relative Khovanov-Jacobsson cycle or class is either clear from context or is unlikely to provide any resistance to an argument. Thus, we often omit it from our notation, writing $\ckj_\Sigma$ and $\kj_\Sigma$ instead of $\ckj_{\Sigma,D}$ and $\kj_{\Sigma,D}$.
	
\end{section}

%
%

\begin{section}{{Calculating Relative Khovanov-Jacobsson Classes}} \label{crkjc}
	
	In this section, we provide a general commentary on calculating relative Khovanov-Jacobsson classes. We give two classifications: for surfaces bounding the unlink and for surfaces obtained by Seifert's algorithm. We discuss techniques for determining nontriviality of enhanced state cycles.
	
	\subsection{{Calculating Induced Maps}} \label{crkjc_techniques}

	Calculating the relative Khovanov-Jacobsson class of a link cobordism $\Sigma\: \emptyset \to L$ depends entirely on calculating the cycle $\ckj_\Sigma$, which in turn requires a thorough understanding of the chain map $\ckh(\Sigma)$. As mentioned in Section \ref{bkh_maps}, the chain maps defining $\ckh(\Sigma)$ can be found in numerous locations throughout the literature. For an example calculation using these maps, see Figure \ref{fig_946left} at the end of the paper.
	
	Although general calculations require a thorough understanding of $\ckh(\Sigma)$, there are extreme cases that can be understood from a theoretical standpoint, as in the following two theorems.
	
	\subsection{{Calculations for Surfaces Bounding the Unlink}} \label{crkjc_unlink}
	Here, we classify relative Khovanov-Jacobsson classes for surfaces bounding the unlink, as in the following theorem.
	
	\stepcounter{theorem}
	\begin{theorem} \label{crkj_theorem1} 
		Let $U$ be the $n$-component unlink and $D$ its crossingless diagram. The relative Khovanov-Jacobsson class for a connected link cobordism $\Sigma: \emptyset \to U$ is determined by its genus:
		\begin{itemize}
			\item[\textnormal{(a)}] if $g(\Sigma) = 0$, then $\ckj_{\Sigma,D}$ is, up to sign, a $pqr$-chain on the components of $D$;
			\item[\textnormal{(b)}] if $g(\Sigma) = 1$, then $\ckj_{\Sigma,D}$ is, up to sign, twice the all $x$-label $\sigma_x$;
			\item[\textnormal{(c)}] if $g(\Sigma) \geq 2$, then $\ckj_{\Sigma,D} = 0$.
		\end{itemize}
	\end{theorem}
	\begin{proof}
		The lemma follows from an analysis on the bigrading to which $\ckj_\Sigma$ belongs. We first determine the support of the chain complex $\ckh(D)$. As $D$ is a crossingless diagram, the chain complex $\ckh(D)$ is supported in the $0$-homological grading. To find the support of the quantum grading, we consider the enhanced states of $D$, which generate the chain groups of $\ckh(D)$. The quantum grading of an enhanced state $\alpha$ for the crossingless diagram $D$ of $U$ is given by 
			\begin{equation}\label{q1} q(\alpha) = v_+(\alpha) - v_-(\alpha)\end{equation}
		As $D$ is an $n$-component crossingless diagram, we also have 
			\begin{equation}\label{q2} n = v_+(\alpha) + v_-(\alpha)\end{equation}
		Neither of $v_\pm(\alpha)$ can be negative, so it follows that the chain group $\ckh(D)$ is supported in the $(0,j)$ grading for $-n \leq j \leq n$.
		
		To determine the exact bigrading for $\ckj_\Sigma$, recall that the induced map $\ckh(\Sigma)$ is a $(0, \chi(\Sigma))$ bigraded chain map, taking $1 \in \ckh^{0,0}(\emptyset)$ to a single quantum grading, $q(\ckh(\Sigma)(1)) = \chi(\Sigma)$. Thus, $\ckj_\Sigma$ belongs to the chain group $\ckh^{0,\chi(\Sigma)}(D)$, whereby it can be expressed as a linear combination of enhanced states from the same bigrading:
			$$\hspace{150pt}\ckj_\Sigma = \sum_{i=1}^m a_i \alpha_i\hspace{50pt} a_i \in \Z, \hspace{5pt} q(\alpha_i) = \chi(\Sigma).$$
		Furthermore, equations (1) and (2) combine with this quantum grading to show $v_+(\alpha_i) = 1 - g(\Sigma)$. 
		
		For $g(\Sigma) \geq 2$, the grading $q(\ckj_\Sigma) = 2 - 2g(\Sigma) - n < -n$ is outside the support, implying (c).
		
		For $g(\Sigma) = 1$, the enhanced states at this grading must satisfy $v_+(\alpha_i) = 0$. So only the all $x$-label  $\sigma_x$ is valid, and $\ckj_\Sigma = a\sigma_x$ for some $a \in \Z$. To establish (b), we need only determine the value of the coefficient. Consider the composition $C \circ \Sigma$, where $C$ is the surface that caps off each boundary component of $\Sigma$ with a disk. Note that $C \circ \Sigma$ is a closed, connected surface with genus $1$, so by the characterization of absolute Khovanov-Jacobsson classes (Theorem \ref{tanakarasmussen} \!), we have
			$$\ckh(C)(\ckj_\Sigma) = \ckh(C)(\ckh(\Sigma)(1)) = \ckh(C \circ \Sigma)(1) = \ckj_{C \circ \Sigma} = \pm2.$$
		Moreover, $C$ induces the map $\otimes_n \varepsilon$, where $\varepsilon = \{ 1 \mapsto 0; x \mapsto 1\}$,  on the enhanced states of $D$, so 
			$$\ckh(C)(\ckj_\Sigma) = \ckh(C)(a\alpha_-) = a.$$
		We conclude that $\ckj_\Sigma = \pm2\alpha_-$, establishing (b).
		
		For $g(\Sigma) = 0$, the only enhanced states at this grading are those consisting of a single $1$-label. In order to conclude (a), we must show that $a_i = \pm1$ and that all the $a_i$ have the same sign. Enumerate the components $U_i$ of $U$, and consistently carry over this enumeration to the components $D_i$ of of the crossingless diagram $D$. Let $\alpha_i$ be the enhanced state with $1$-labeled $D_i$ and $x$-labeled $D_{j \neq i}$.
		
		To see that each $a_i = \pm1$, consider $C_j \circ \Sigma$, where $C_j$ is the surface that caps off $U_j$ with a punctured torus and the remaining boundary component $U_{i \neq j}$ with a disk. Note $C_j \circ \Sigma$ is a closed, connected surface with genus $1$, so by the characterization of absolute Khovanov-Jacobsson classes,
			$$\ckh(C_j)(\ckj_\Sigma) = \ckh(C_j)(\ckh(\Sigma)(1)) = \ckh(C_j \circ \Sigma)(1) = \pm2.$$
		Moreover, $C_j$ induces the map $\varepsilon$ (defined above) on each $D_i$ and $\{1 \mapsto 2; x \mapsto 0\}$ on $D_j$. It follows that $\ckh(C_j)(\alpha_j) = 2$ and $\ckh(C_{i \neq j})(\alpha_j) = 0$, so applying $C_j$ to $\ckj_{\Sigma}$, we have 
			$$\ckh(C_j)(\ckj_\Sigma) = \ckh(C_j) \bigg( \sum_{i=1}^m a_i \alpha_i \bigg) = \sum_{i=1}^m a_i\ckh(C_j)(\alpha_i) = a_j\ckh(C_j)(\alpha_j) = 2a_j$$
		So $a_j = \pm 1$, as desired. 
		
		\begin{figure}[!ht]
			\includegraphics[scale=.3]{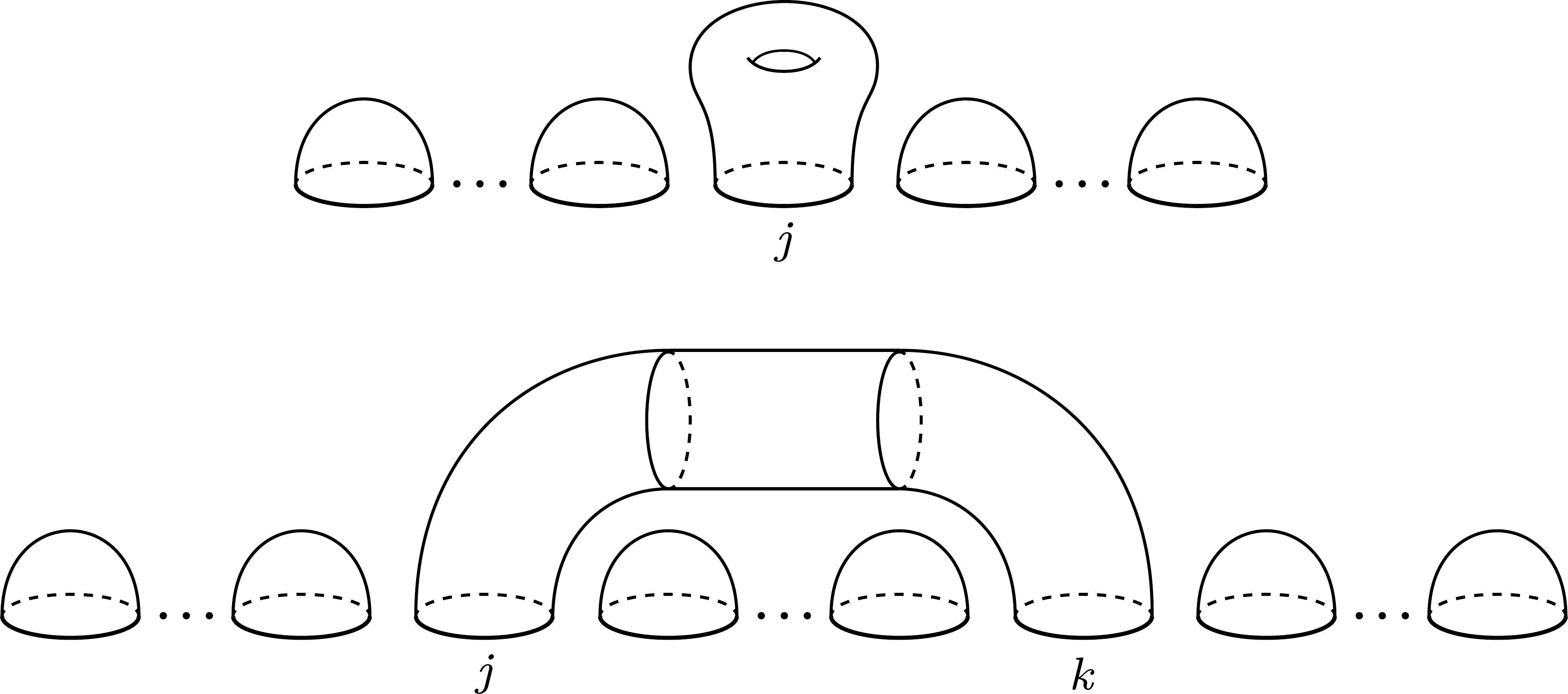}
			\caption{From left to right: the surfaces $C_j$ and $C_{j,k}$.}
			\label{fig_unlink_maps}
		\end{figure}
		
		To see each $a_i$ has the same sign, consider $C_{j,k} \circ \Sigma$, where $C_{j,k}$ is the surface that caps off each $U_{i \neq j,k}$ with a disk and connects the remaining two $U_j$ and $U_k$ with an anulus. A nearly identical calculation shows $a_j + a_k = \pm 2$ for each $j \neq k$, which holds only when $a_j = \pm1$ and $a_k = \pm1$ having the same sign. Thus, $\ckj_\Sigma$ is the desired $pqr$-chain on $D$.
	\end{proof}

	\begin{remark}
		Theorem \ref{crkj_theorem1} assumes that $\Sigma$ is a connected surface. In the case that there are multiple components, we may draw on the recent result of \cite{gujrallevine20}, which shows that for split links, the induced map $\ckh(\Sigma)$ is determined by the maps induced between components of $\Sigma$, independent of their potential linking. Thus, this theorem can be applied to each component of $\Sigma$ individually.
	\end{remark}
	
	\subsection{{Calculations for Surfaces Obtained from Seifert's Algorithm}} \label{crkjc_seifert}
	Here, we classify the relative Khovanov-Jacobsson classes of surfaces obtained by applying Seifert's algorithm. 
	
	Recall that to a diagram $D$ of an oriented link, Seifert's algorithm associates an oriented, compact surface $\Sigma_D \subset S^3$ with boundary $L$ by applying the following procedure: produce a smoothing $\sigma_D$ of $D$ according to its orientation\!
		\footnote{A diagram is smoothed according to its orientation if each positive (negative) crossing is $0$-smoothed ($1$-smoothed).}
	\!, resulting in a collection of planar \textit{Seifert circles}, that can be oriented accordingly; each Seifert circle bounds an appropriately oriented \textit{Seifert disk} in the plane, and these disks can be situated in $S^3$ so that an outermost disk is below any disk it contains; each crossing in $D$ connects two Seifert disks and is realized as the boundary of a twisted band between the disks. The collection of disks and bands is an oriented surface $\Sigma_D$, called a \textit{Seifert surface}. There is a movie description of $\Sigma_D$ inherent to the structure of this algorithm:
	\begin{itemize}
		\item[(1)] birth each Seifert disk in ascending order of their height (lowest, or outermost, first);
		\item[(2)] for each negative crossing, perform a \textit{negative twist-and-glue} move (Figure \ref{fig_twistnglue_negative} \!) consisting of a single negative Reidemeister I and a single Morse saddle;
		\item[(3)] for each positive crossing, perform a \textit{positive twist-and-glue} move (Figure \ref{fig_twistnglue_positive} \!) consisting of a single positive Reidemeister I and a single Morse saddle.
	\end{itemize}
	Following this description, decompose the Seifert surface $\Sigma_D = \Sigma_3 \circ \Sigma_2 \circ \Sigma_1$, where each $\Sigma_i$ is the surface corresponding to the $i$-th \textit{scene} of the movie listed above.
	
	\begin{figure}[!ht]
		\includegraphics[scale=.4]{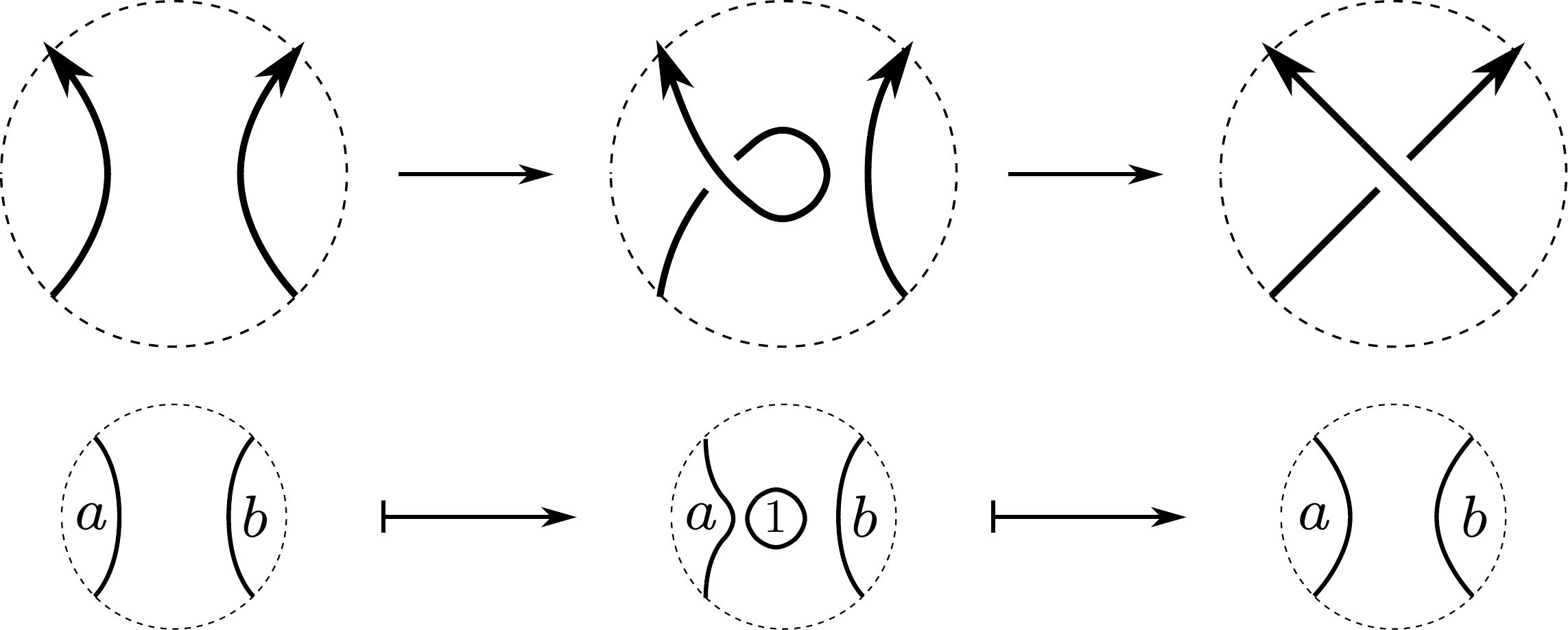}
		\caption{The negative twist-and-glue move, along with its induced chain map.}
		\label{fig_twistnglue_negative}
	\end{figure}
	\begin{figure}[!ht]
		\includegraphics[scale=.4]{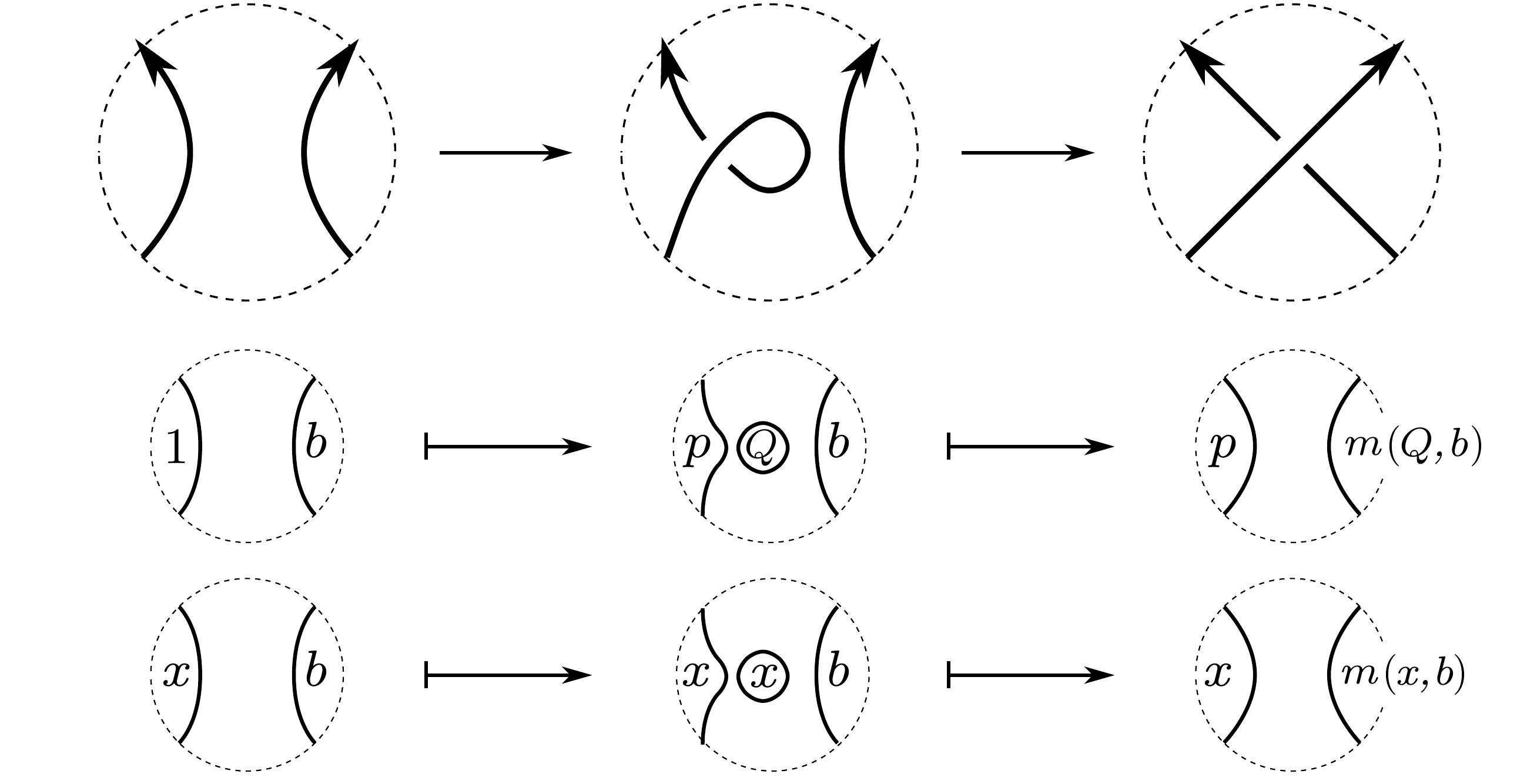}
		\caption{The positive twist-and-glue move, along with its induced chain map.}
		\label{fig_twistnglue_positive}
	\end{figure}

	\begin{remark}
		The cycle $\ckj_{\Sigma_D}$ produced by the above movie description of $\Sigma_D$ is fairly tame, with the decomposition $\Sigma_D = \Sigma_3 \circ \Sigma_2 \circ \Sigma_1$ allowing us to take this calculation in stages. First, note that $\ckh(\Sigma_1)(1)$ is the all-$1$ label $\sigma_1$ on the orientation-induced smoothing $\sigma_D$. Second, each negative twist-and-glue move induces the identity map, so $\ckh(\Sigma_2 \circ \Sigma_1)(1)$ is again the all-$1$ label $\sigma_1$ on $\sigma_D$. Third, each positive twist-and-glue move does not change the underlying smoothing, but instead affects only the labels. We conclude that 
			$$\ckj_{\Sigma_D} = \ckh(\Sigma_3 \circ \Sigma_2 \circ \Sigma_1)(1) = \ckh(\Sigma_3)\big(\ckh(\Sigma_2 \circ \Sigma_1)(1) \big) = \ckh(\Sigma_3)(\sigma_1)$$
		is a sum of enhanced states on the smoothing $\sigma_D$, with labels determined by the positive twist-and-glue moves. 
	\end{remark}
	
	\begin{remark}
		The positive crossings are all $0$-smoothed, so each positive twist-and-glue move occurs on an edge of $\Gamma_0$ (see Section \ref{bkh_elements} \!), whereby, the map $\ckh(\Sigma_3)$ is described by how it acts on the labels of vertices in $\Gamma_0$. The vertices of the state graph $\Gamma_{\sigma_D}$ that are not in the subgraph $\Gamma_0$ are $1$-labeled, and are not affected by the positive twist-and-glue moves. The following theorem shows that the homotopy type\ of the subgraph $\Gamma_0$ determines $\ckj_{\Sigma_D}$.
	\end{remark}
	
	\begin{theorem} \label{crkj_theorem2}
		The relative Khovanov-Jacobsson cycle $\ckj_{\Sigma_D}$ of a Seifert surface $\Sigma_D$ is determined by the first Betti number of the subgraph $\Gamma_0$ of the state graph $\Gamma_{\sigma_D}$ corresponding to the orientation-induced smoothing $\sigma_D$ of the chosen diagram $D$. Explicitly, for connected $\Gamma_0$ we have:
		\begin{itemize}
			\item[\textnormal{(a)}] if $b^1(\Gamma_0) = 0$, then $\ckj_{\Sigma_D}$ is a $pqr$-chain on the vertices of $\Gamma_0$  \textnormal{(}the sign of each summand in the $pqr$-chain is determined by the parity of the distance of its unique $1$-labeled vertex from a predetermined vertex $v_0$ of $\Gamma_0$\textnormal{)}.
			\item[\textnormal{(b)}] if $b^1(\Gamma_0) = 1$, then $\ckj_{\Sigma_D}$ is twice the enhanced state having $x$-labels on the vertices of $\Gamma_0$;
			\item[\textnormal{(c)}] if $b^1(\Gamma_0) \geq 2$, then $\ckj_{\Sigma_D} = 0$.
		\end{itemize}
	\end{theorem}
	\begin{proof}		
		Suppose first that $b^1(\Gamma_0) = 0$. Choose a vertex $v_0$ in $\Gamma_0$. Partially order the edge set $E_0 = \{e_1, e_2, \dots, e_k\}$ by their distance from $v_0$, and situate $\Sigma_3$ so the positive twist-and-glue moves are done in this order. The edge $e_1$ connects $v_0$ to some $v_1$, both of which are $1$-labeled. The positive twist-and-glue move associated to $e_1$ produces a $pqr$-chain on these vertices. We claim that successively applying the remaining maps induced by the edges $e_2, \dots, e_k$ will produce a $pqr$-chain on $\Gamma_0$. Inducting on the edge set, suppose the edge $e_i$ produces a $pqr$-chain on the vertices $v_0, \dots, v_i$ incident to one of $e_1, \dots, e_i$. The edge $e_{i+1}$ connects one of these vertices, say $v_j$, to a new, $1$-labeled vertex $v_{i+1}$. To see that the map induced by $e_{i+1}$ produces a $pqr$-chain on the vertices $v_1, \dots, v_i, v_{i+1}$, note that when $v_j$ is $x$-labeled, the positive twist-and-glue map incorporates $v_i$ into the new $pqr$-chain as another $x$-labeled loop, and when $v_j$ is $1$-labeled, the map produces the pair of unique summands in the new $pqr$-chain with $1$-labels on either $v_{i+1}$ or $v_j$. The sign of the unique summand with $1$-labeled $v_{i+1}$ is determined by the sign of $v_j$: the sign indicated in the positive twist-and-glue map (Figure \ref{fig_twistnglue_positive} \!) is either inherited by $v_{i+1}$ (when $v_j$ is positive) or is doubly-negated (when $v_j$ is negative). Consequently, the sign is determined by the parity of the distance of $v_{i+1}$ from $v_0$. It is negative when the distance is odd and positive when even. We conclude that $\ckj_{\Sigma_D}$ has the desired description.
		
		Suppose $b^1(\Gamma_0) = 1$. Choose a maximal tree $G < \Gamma_0$ and let $e$ be the unique remaining edge. Situate $\Sigma_3$ so that the positive twist-and-glue maps associated to edges of $G$ are done before $e$. Applying the previous paragraph to $G$, we produce a $pqr$-chain on $G$. Since $e$ connects two vertices of $G$, the map it induces is applied to a pair of loops already contained in the $pqr$-chain. When these loops are both $x$-labeled, the summand vanishes. Otherwise, one loop is $x$-labeled and the other $1$-labeled, and the result is the enhanced state with all $x$-labeled $\Gamma_0$. There are exatly two summands in the $pqr$-chain where these loops are not both $x$-labeled, so the result of applying $e$ is twice the all $x$-labeled $\Gamma_0$. The sign is the product of the signs of these summands.
		
		When $b^1(\Gamma_0) \geq 2$, apply the previous paragraph and record an additional edge that is not in $G$. Applying the positive twist-and-glue map associated to this edge will kill the enhanced state we have just produced.
	\end{proof}
	
	\begin{remark}
		In the case that $\Gamma_0$ is not connected, the theorem may be individually applied to each component; the positive twist-and-glue moves associated to different components act on labels for distinct Seifert circles (if they shared a Seifert circle, the two components would share a vertex).
	\end{remark}
	
	\subsection{{Nontriviality of Relative Khovanov-Jacobsson Classes}} \label{crkjc_nontriviality}
	In the following sections, we will use the nontriviality of relative Khovanov-Jacobsson classes to obstruct certain properties of surfaces bounding a link. It is then important to establish conditions for determining when cycles in the Khovanov chain complex represent nontrivial homology classes. Conditions specific to enhanced state cycles are discussed in \cite{elliott09_1} and \cite{elliott09_2}. Unfortunately, these conditions are not exhaustive. Moreover, relative Khovanov-Jacobsson cycles are often linear combinations of enhanced states, with each term not necessarily constituting a cycle. Further development of these conditions would be promising for future applications of relative Khovanov-Jacobsson classes.
	
	One method that has proven useful for determine the nontriviality of some cycles is the following: if it is difficult to determine whether the relative Khovanov-Jacobsson cycle of a link cobordism $\Sigma\: \emptyset \to L$ is nontrivial, we can extend $\Sigma$ by a second link cobordism $\Sigma'\:L \to L'$ and hope that the nontriviality of $\kj_{\Sigma' \circ \Sigma}$ is simpler to determine in $\kh(L')$. Because $\kj_{\Sigma' \circ \Sigma} = \kh(\Sigma')(\kj_\Sigma)$ and $\kh(\Sigma')$ is a homomorphism, a nontrivial $\kj_{\Sigma' \circ \Sigma}$ necessitates a nontrivial $\kj_\Sigma$. 
	
	One link cobordism $L \to L'$ we will use is called a \textit{trim cobordism}. This link cobordism removes a single crossing from the diagram for $L$, thereby reducing the complexity of the Khovanov homology group in which nontriviality is being determined. A crossing is removed by attaching a $1$-handle and undoing the resulting kink with a Reidemeister I move. This is illustrated for a right-handed crossing in Figure \ref{fig_trim_positive} and for a left-handed crossing in Figure \ref{fig_trim_negative}. The $1$-handle can be attached on either the right or left of the crossing; the computations are symmetric, so we have only shown the left-sided attachment. The map induced by a trim cobordism is called a \textit{trim map}.
	
	\begin{remark}
		The reader may be alarmed that the $1$-handle used to perform a trim cobordism may produce a nonorientable link cobordism. Because trim maps are only used for determining nontriviality, the orientability of the link cobordism $L \to L'$ is unimportant: nonoriented Morse saddles still induce chain maps. A choice of orientation for the link $L'$ simply determines the grading of the chain map, and therefore, will not effect the nontriviality of the cycle.
	\end{remark}
	
	\begin{remark}
		Trim maps \textit{regularly} produce cycles representing the trivial homology class. Only clever and careful use of trim maps provide useful insights!
	\end{remark}
	
\end{section}

%
%

\begin{section}{{Obstructing Sliceness of Knots}} \label{osk}
	
	\subsection{{Techniques for Obstructing Sliceness}} \label{osk_slice}
	In certain cases, relative Khovanov-Jacobsson classes can be used to determine the sliceness of a knot. In particular, by implementing the characterization of absolute Khovanov-Jacobsson classes in Theorem \ref{tanakarasmussen} \!, we obtain the following result.
	
	\begin{theorem} \label{osk1}
		For a slice knot $K$, the relative Khovanov-Jacobsson class of a link cobordism $\Sigma\: \emptyset \to K$ is nontrivial if $g(\Sigma) \leq 1$. In particular, slices of a slice knot have nontrivial relative Khovanov-Jacobsson classes.
	\end{theorem}
	\begin{proof}
		First, suppose $g(\Sigma) = 0$. Produce an oriented, genus $1$, closed surface $T = -(\Sigma \# T^2) \circ \Sigma$. By the characterization of absolute Khovanov-Jacobsson classes, $\kh(T)(1) = 2$. Therefore,
			$$2 = \kh(T)(1) = \big(\kh(-(\Sigma \# T^2)) \circ \kh(\Sigma)\big)(1)= \kh(-\Sigma \# T^2)(\kj_\Sigma).$$
		A trivial $\kj_\Sigma$ kills the right-most term, producing a contradiction. The proof with $g(\Sigma) = 1$ is done similarly; compose $\Sigma$ with a slice of $K$ to yield an identical calculation.
	\end{proof}

\clearpage

\begin{figure}[!ht]
	\includegraphics[scale=.4]{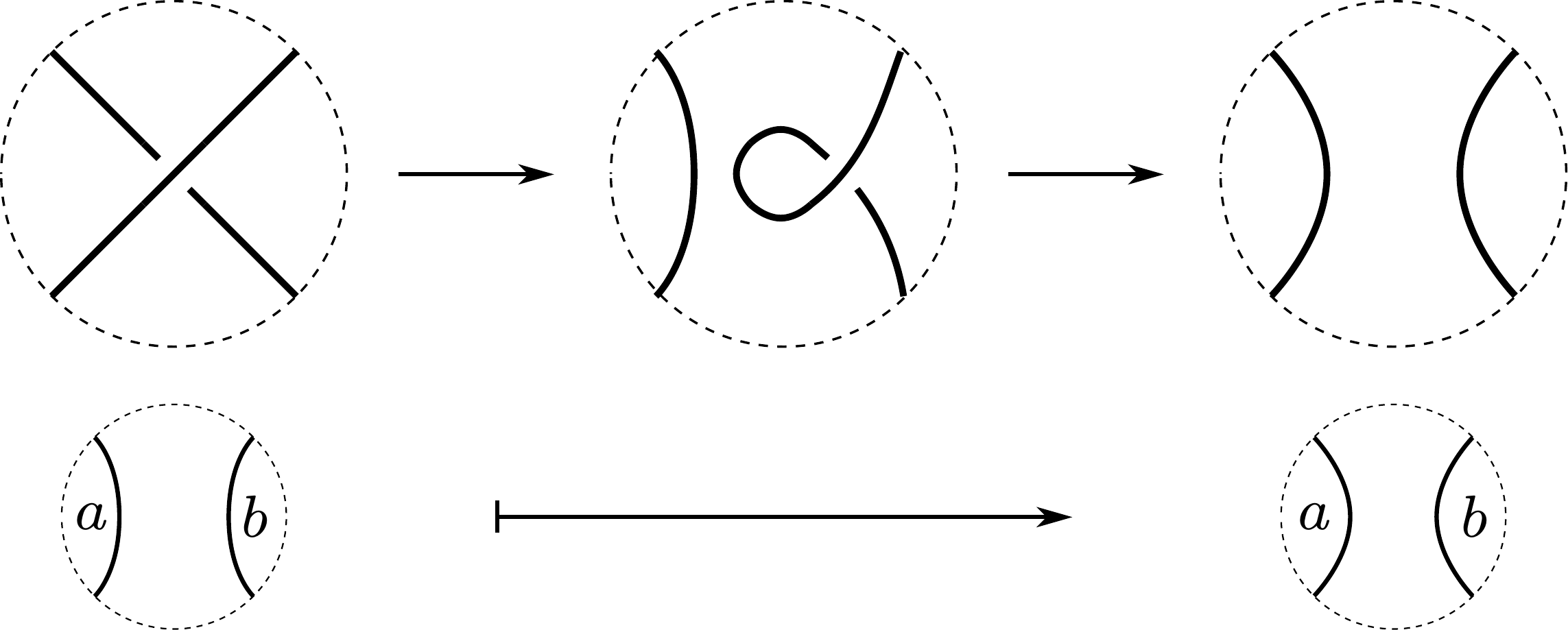}
	\caption{The left trim map on a right-handed crossing.}
	\label{fig_trim_positive}
\end{figure}

\begin{figure}[!ht]
	\includegraphics[scale=.4]{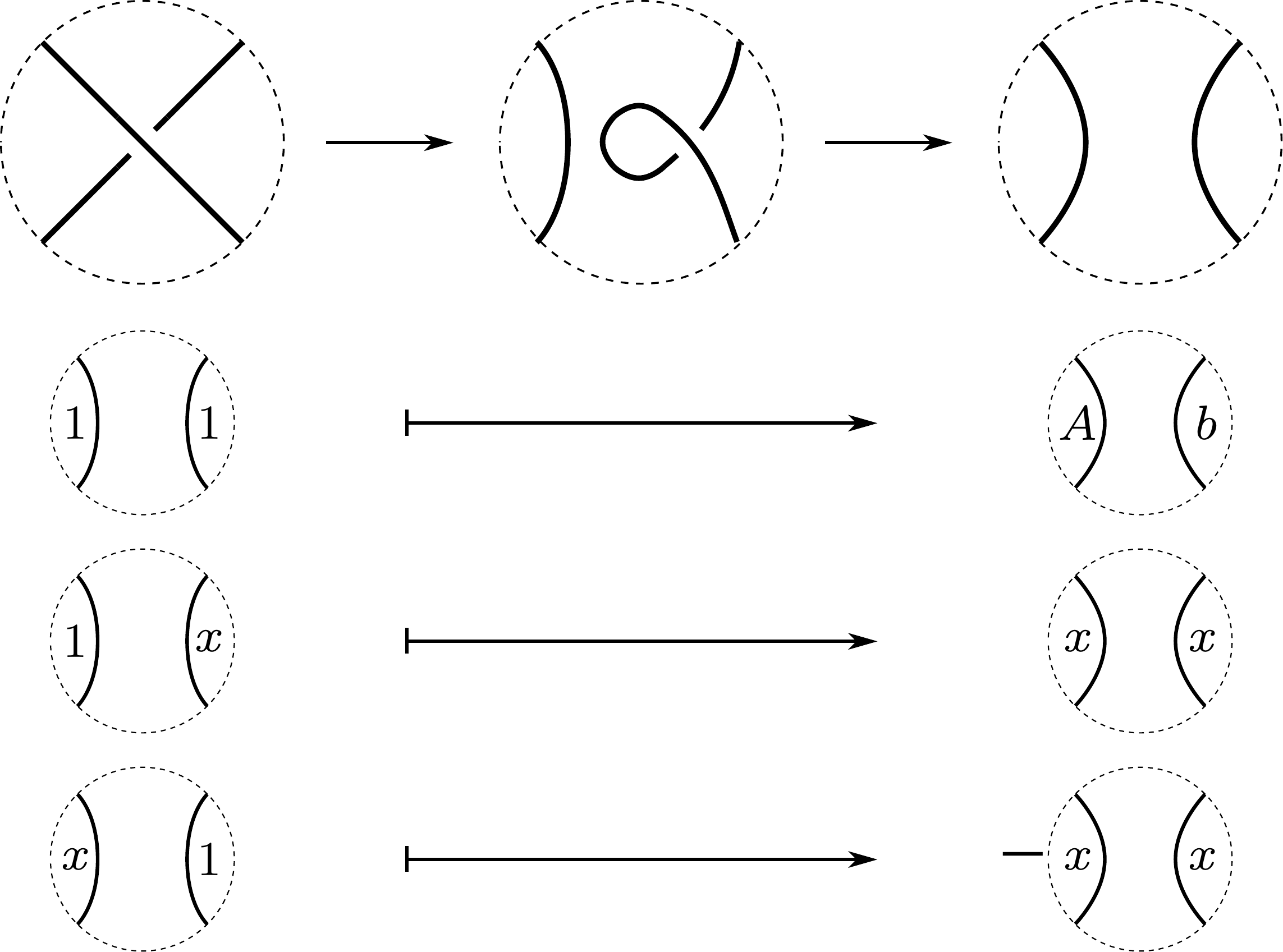}
	\caption{The left trim map on a left-handed crossing.}
	\label{fig_trim_negative}
\end{figure}
	
	\subsection{{Obstructing Sliceness of Pretzel Knots}} \label{osk_pretzel}
	Here, we use relative Khovanov-Jacobsson classes to obstruct the sliceness of certain three stranded pretzel knots. The sliceness of all odd three-stranded pretzel knots was determined in \cite{greenejabuka11}, and for the family we consider, the signature provides a more direct obstruction to sliceness. Nevertheless, this result indicates the potential for relative Khovanov-Jacobsson classes to obstruct sliceness more generally.
	
	\begin{theorem} \label{osk_theorem1}
		For $p,q,r$ odd and of the same sign, $P(p,q,r)$ is not slice.
	\end{theorem}
	\begin{proof}
		A knot is slice if and only if its mirror is slice; since the mirror of $P(p,q,r)$ is $P(-p,-q,-r)$, it suffices to prove the case where $p$, $q$, and $r$ are all negative. We will show that the Seifert surface $\Sigma_D$ associated to the standard diagram $D$ in Figure \ref{fig_slice_obstruction} has trivial relative Khovanov-Jacobsson, whereby Theorem \ref{osk1} implies $P(p,q,r)$ is not slice. To show that $\Sigma_D$ has trivial relative Khovanov-
		
		\begin{figure}[!ht]
			\includegraphics[scale=.4]{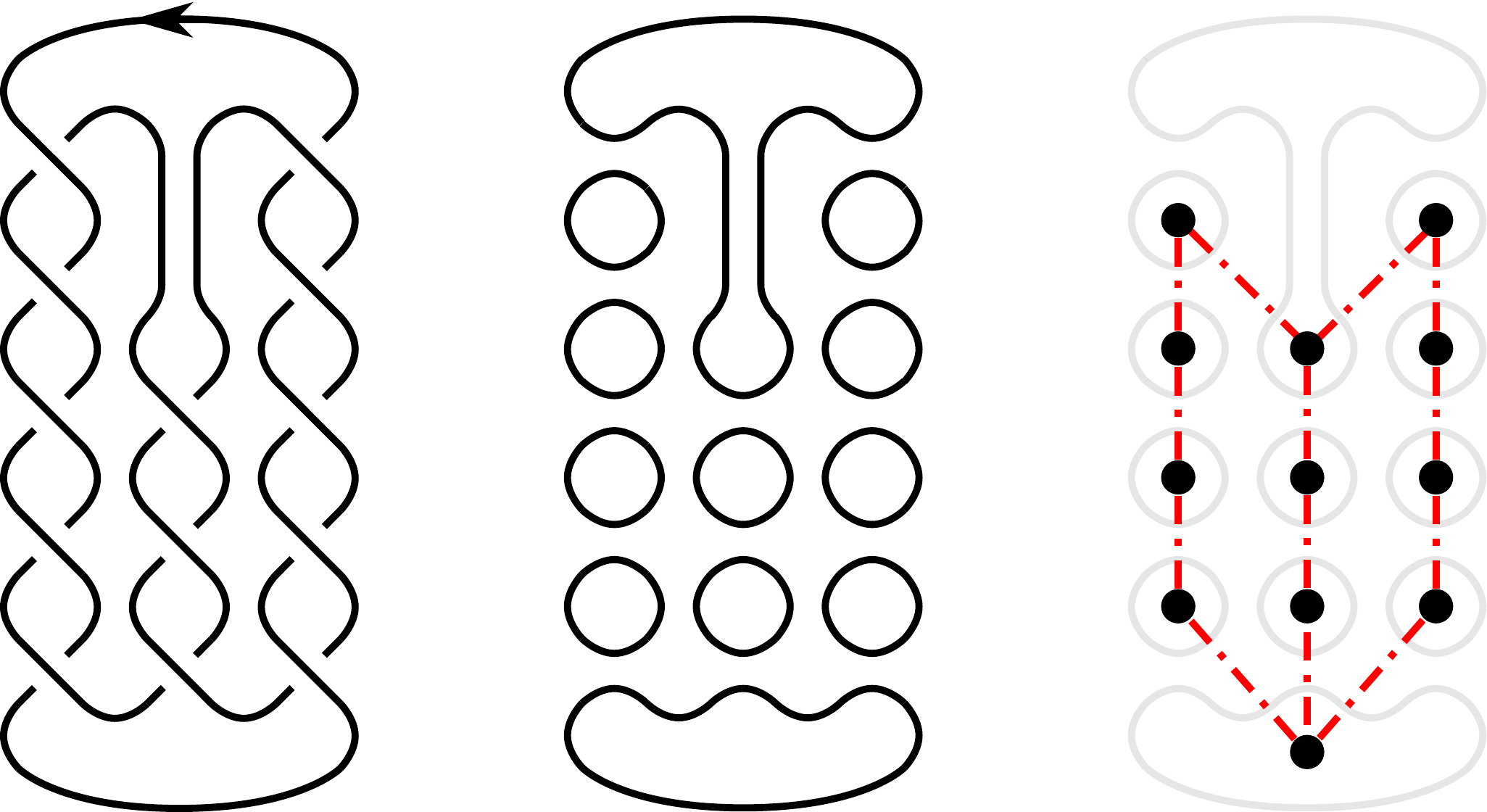}
			\caption{An oriented $P(-5,-3,-5)$; the orientation-induced smoothing $\sigma_D$; the subgraph $\Gamma_0$ of the state graph $\Gamma_{\sigma_D}$, having first Betti number $2$.}
			\label{fig_slice_obstruction}
		\end{figure}
		
		\noindent Jacobsson class, we apply Theorem \ref{crkj_theorem2} \!, which characterizes the relative Khovanov-Jacobsson cycles for Seifert surfaces according to the Betti number of the subgraph $\Gamma_0$ of their state graph $\Gamma_{\sigma_D}$. Either orientation induces the $0$-state of $D$ whose subgraph $\Gamma_0 = \Gamma_{\sigma_D}$ has first Betti number $2$. Thus, the relative Khovanov-Jacobsson cycle for $\Sigma_D$ is trivial. We conclude $P(p,q,r)$ is not slice.
	\end{proof}

	\begin{remark}
		In an attempt to reprove the result of \cite{greenejabuka11}, we could consider the case where one parameter $q$ is positive. In this case, $b^1(\Gamma_{\sigma_D}) = 1$ so $\kj_{\Sigma}$ is twice the enhanced state having $x$-labels on the vertices of $\Gamma_0$ and $1$-labels otherwise. It is unclear if this enhanced state represents a nontrivial Khovanov homology class. On the other hand, $\kj_\Sigma$ is nontrivial if $q = -p$ or $q = -r$ because $P(p,-p,r)$ and $P(p,-r,r)$ are both slice. We suspect that this is the case for all $P(p,q,r)$ with positive $q$ because the graph structure of $\Gamma_0$ is unchanged by the value of the parameters.
	\end{remark}

	\begin{remark}
		Odd parameters for a pretzel link always produce a knot. Using one even parameter will also be a knot. Unfortunately, in this case, the Seifert surface rarely has Euler characteristic $0$ or $-1$, so the above argument is not applicable to this situation.
	\end{remark}

	
\end{section}

%
%

\begin{section}{{Obstructing Boundary-Preserving Isotopy Classes}} \label{obpic}
	
	\subsection{{Techniques for Obstructing Boundary-Preserving Isotopy Classes}} \label{obpic_techniques}
	Recall that relative Khovanov-Jacobsson classes are invariants of the boundary-preserving isotopy class of a link cobordism: two link cobordisms $\Sigma_{0,1}\:\emptyset \to L$ whose relative Khovanov-Jacobsson classes are distinct, $\kj_{\Sigma_0} \neq \kj_{\Sigma_1}$, are not isotopic relative to their boundary. To obstruct this isotopy, we show
		$$\mathcal{C}(\Sigma_0, \Sigma_1) := \pm\ckj_{\Sigma_0} \pm \ckj_{\Sigma_1}$$
	represent nontrivial Khonvanov homology classes.
	
	\subsection{{Local Knottedness}} \label{obpic_local} 
	Given a link cobordism $\Sigma\: \emptyset \to L$ and a knotted $2$-sphere $S$, we say $\Sigma \# S$ is \textit{locally knotted}. This operation generally changes the boundary-preserving isotopy class of $\Sigma$, and as there is a plethora of knotted $2$-spheres with which we can perform this operation, the study of boundary-preserving isotopy classes is more rewarding when this operation is omitted. For a knot $K$, let $\bar{\Sigma(K)}$ denote the set of equivalence classes of surfaces bounding $K$ under the equivalence relation generated by boundary-preserving isotopy and connect summing with a knotted $2$-sphere. For a link cobordism $\Sigma\: \emptyset \to K$ we let $\Sigma(K)$ denote its equivalence class in $\bar{\Sigma(K)}$.
	
	Recent work has drawn attention toward understanding $\bar{\Sigma(K)}$ for classes of slice knots; see \cite{juhaszzemke20}, \cite{conwaypowell20}, and \cite{millerpowell20}. The latter of these posed the question: \textit{can we determine the value of $|\bar{\Sigma(K)}|$ for some nontrivial knot $K$, or at least whether $|\bar{\Sigma(K)}| < \infty$}. The following Theorem and Corollary indicate that relative Khovanov-Jacobsson classes may provide some aid in calculating lower bounds for this value.
	
	\begin{theorem}\label{obpic_theorem1}
		Relative Khovanov-Jacobsson classes do not detect local knottedness.
	\end{theorem}
	\begin{proof}
		For a link cobordism $\Sigma\:\emptyset \to L$ and $2$-knot $S$, we wish to show $\kj_\Sigma = \kj_{\Sigma \# S}$. Let $E$ and $F$ be the disks removed from $\Sigma$ and $S$, respectively, to form $\Sigma \# S$. This allows us to write $\Sigma \# S$ as a composition of cobordisms $(\Sigma \setminus E) \circ (S \setminus F)$. In forming the connect sum, the disk $E$ is replaced by $S \setminus F$. As both of these are genus $0$ surfaces bounding the unknot, Theorem \ref{crkj_theorem1} implies they induce the same map on Khovanov homology, whereby we have
		$$\kj_\Sigma = \kh_{(\Sigma \setminus E) \circ E}(1) = \big(\kh_{\Sigma \setminus E} \circ \kh_E\big)(1) = \big(\kh_{\Sigma \setminus E} \circ \kh_{S \setminus F}\big)(1) = \kh_{(\Sigma \setminus E) \circ (S \setminus F)}(1) = \kj_{\Sigma \# S}$$
		Therefore, local knottedness is not detected by relative Khovanov-Jacobsson classes.
	\end{proof}
	
	Although the statement of this theorem may sound like a shortcoming of relative Khovanov-Jacobsson classes, it actually gives them the power to distinguish classes in $\bar{\Sigma(K)}$. In particular, whenever a pair of surfaces bounding a knot $K$ have distinct relative Khovanov-Jacobsson classes, it must follow that the surfaces represent distinct classes in $\bar{\Sigma(K)}$, lest they contradict the above theorem. More succinctly, we have the following result.
	
	\begin{corollary}\label{obpic_corollary2}
		Relative Khovanov-Jacobsson classes are invariants of classes in $\bar{\Sigma(K)}$ for any $K$.
	\end{corollary}
	
	\subsection{{Obstructing Equivalence of Slice Disks for $\boldsymbol{9_{46}}$}} \label{obpic_946}
	There are two slices $\Sigma_\ell$ and $\Sigma_r$ for the knot $9_{46}$, obtained by first performing either the left (pink) or right (orange) band move, indicated
	\begin{wrapfigure}[5]{r}{7em}
		\centering\vspace{-8pt}
		\hspace{-5pt}\includegraphics[scale=.4]{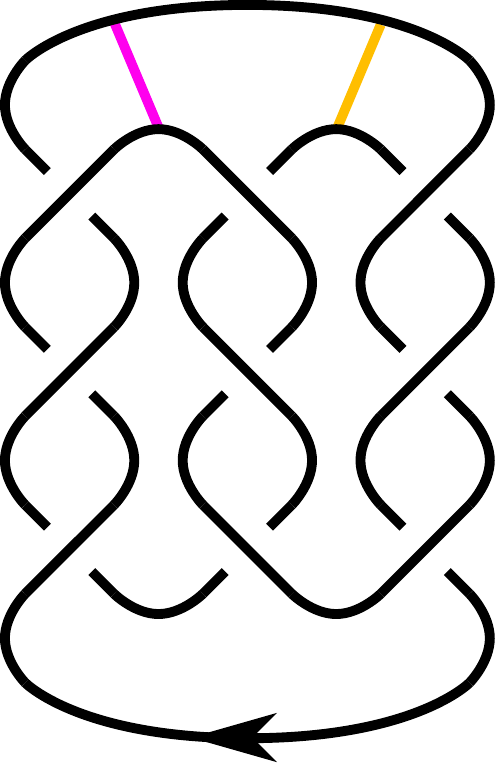}
	\end{wrapfigure}
	on the diagram to the right, and then undoing the resulting two-component unlink. It is known that these slices are not isotopic relative to their boundary. In what follows, we verify this fact by showing these slices have distinct relative Khovanov-Jacobsson classes. We go on to show that $\#_k(9_{46})$ has $2^k$ distinct slices, implying $|\bar{\Sigma\big(\#_k(9_{46})\big)}| \geq 2^k$. This fact was also established in \cite{millerpowell20} by examining kernels of inclusion induced maps on Alexander modules.
	
	\begin{theorem}\label{obpic_theorem3}
		The pretzel knot $P(3, -3,3) = 9_{46}$ has at least two distinct slices.
	\end{theorem}
	
	\begin{proof}
		A movie of the surface $\Sigma_r$ is given in Figure \ref{fig_946left} (found at the end of the paper), along with the relative Khovanov-Jacobsson cycle it produces. A symmetric calculation gives the relative Khovanov-Jacobsson cycle for $\Sigma_\ell$. Both cycles are illustrated in Figure \ref{fig_946cycles} \!. We claim that the up-to-sign differences of these cycles $\mathcal{C}(\Sigma_\ell, \Sigma_r)$ represent nontrivial homology classes, and therefore the boundary-preserving isotopy classes of $\Sigma_\ell$ and $\Sigma_r$ are distinguished by their relative Khovanov-Jacobsson classes.
		
		\begin{figure}[!ht]
			\includegraphics[scale=.4]{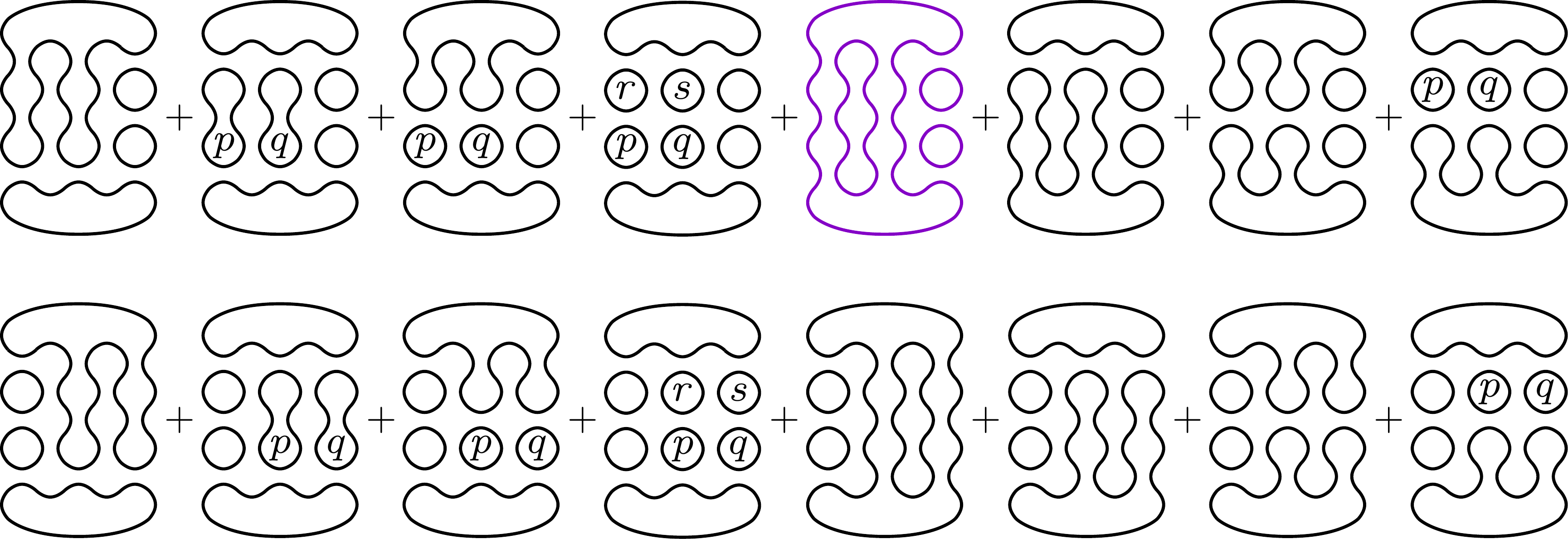}
			\caption{The relative Khovanov-Jacobsson cycles for $\Sigma_\ell$ (top) and $\Sigma_r$ (bottom).}
			\label{fig_946cycles}
		\end{figure}
	
		\begin{remark}
			Although the movie descriptions for these surfaces produce cycles belonging to Khovanov chain complexes associated to the same diagram of $9_{46}$, these chain complexes differ in their enumeration of the diagram's crossings. In retrospect, this will not be important, however, for the general well-being of the reader, we note that these complexes are related by a chain homotopy that only affects the signs of the enhanced states in $\ckj_{\Sigma_\ell}$ and $\ckj_{\Sigma_r}$. There are no common smoothings between these cycles, so these sign changes will not cause any unwanted cancellation.
		\end{remark}
		
		To see $\mathcal{C}(\Sigma_\ell, \Sigma_r)$ represents a nontrivial element in homology, apply three positive trim maps to the crossings in the left column of the diagram, killing any enhanced state having a $1$-smoothing in the left column. There is only one enhanced state whose left column contains only $0$-smoothings; we have highlighted this state (in purple) in $\ckj_{\Sigma_r}$. As a result, the three trim maps produce a single enhanced state. This enhanced state is the all $1$-label of the $1$-state of the trimmed diagram;
		
		\clearpage
		
		\begin{figure}[!ht]
			\includegraphics[scale=.4]{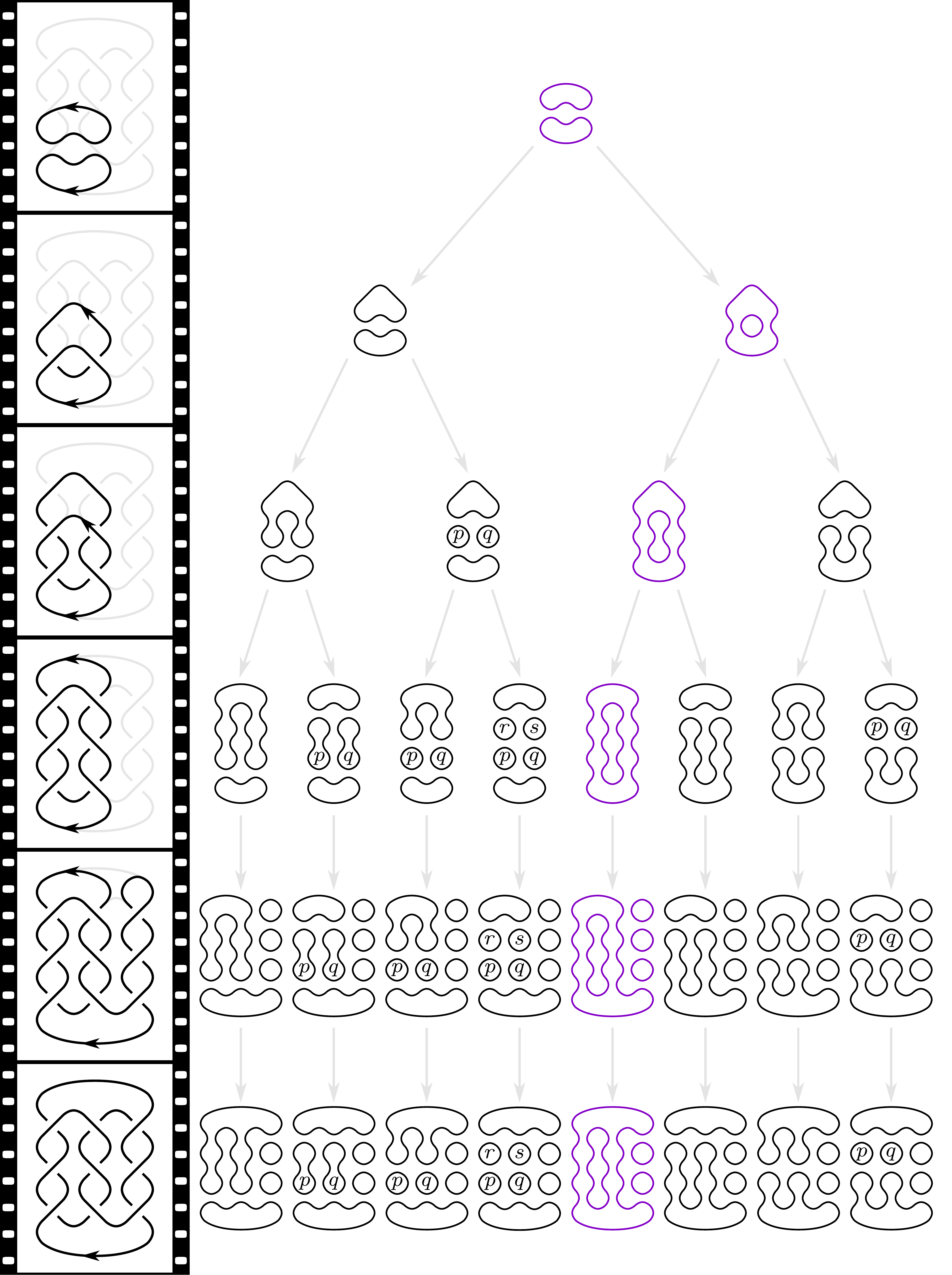}
			\caption{A movie description of the surface $\Sigma_r$ and the calculation of its relative Khovanov-Jacobsson cycle. Highlighted is the element which maintains a left column that is all $0$-smoothed.}
			\label{fig_946left}
		\end{figure}
	
		\clearpage
		
		\noindent an analysis on the differentials entering this (extreme) homological grading shows this enhanced state cannot be a boundary. Details for this analysis can be found in both \cite[Example 2.2]{elliott09_1} and \cite[Proposition 10]{swann10}.
		
		Unpackaging the proof: this enhanced state is the image of $\mathcal{C}(\Sigma_\ell, \Sigma_r)$ under a sequence of trim maps and represents a nontrivial class in the Khovanov homology associated to the trimmed diagram. Thus, the cycles $\mathcal{C}(\Sigma_\ell, \Sigma_r)$ must also represent nontrivial Khovanov homology classes, forcing their relative Khovanov-Jacobsson classes to be distinct. We may then conclude that the slices $\Sigma_\ell$ and $\Sigma_r$ are not isotopic relative to their boundary.
	\end{proof}
	
	\begin{corollary} \label{obpic_corollary4}
		For $n$ odd and $|n| > 1$, the pretzel knot $P(n,-n,n)$ has at least two distinct slices.
	\end{corollary}
	\begin{proof}
		Use a similar diagram with $n$-crossings in each column and alter the above movie descriptions to use $n$-many Reidemeister I and II moves. The relative Khovanov-Jacobsson chain cycles are more complex, however, throughout their calculation, we can follow a unique enhanced state having a left column that is all $0$-smoothed (just as we did in purple throughout Figure \ref{fig_946left} \!). Applying $n$-many positive trim moves to the left column of the diagram will kill all but this unique enhanced state, and, as before, this enhanced state will be the all $1$-label of the $1$-state of the trimmed diagram, which represents a nontrivial Khovanov homology class.
	\end{proof}
	
	\begin{corollary}\label{obpic_corollary5}
		The knot $\#_k (9_{46})$ has at least $2^k$-many distinct slices.
	\end{corollary}
	
	\begin{figure}[!ht]
		\includegraphics[scale=.375]{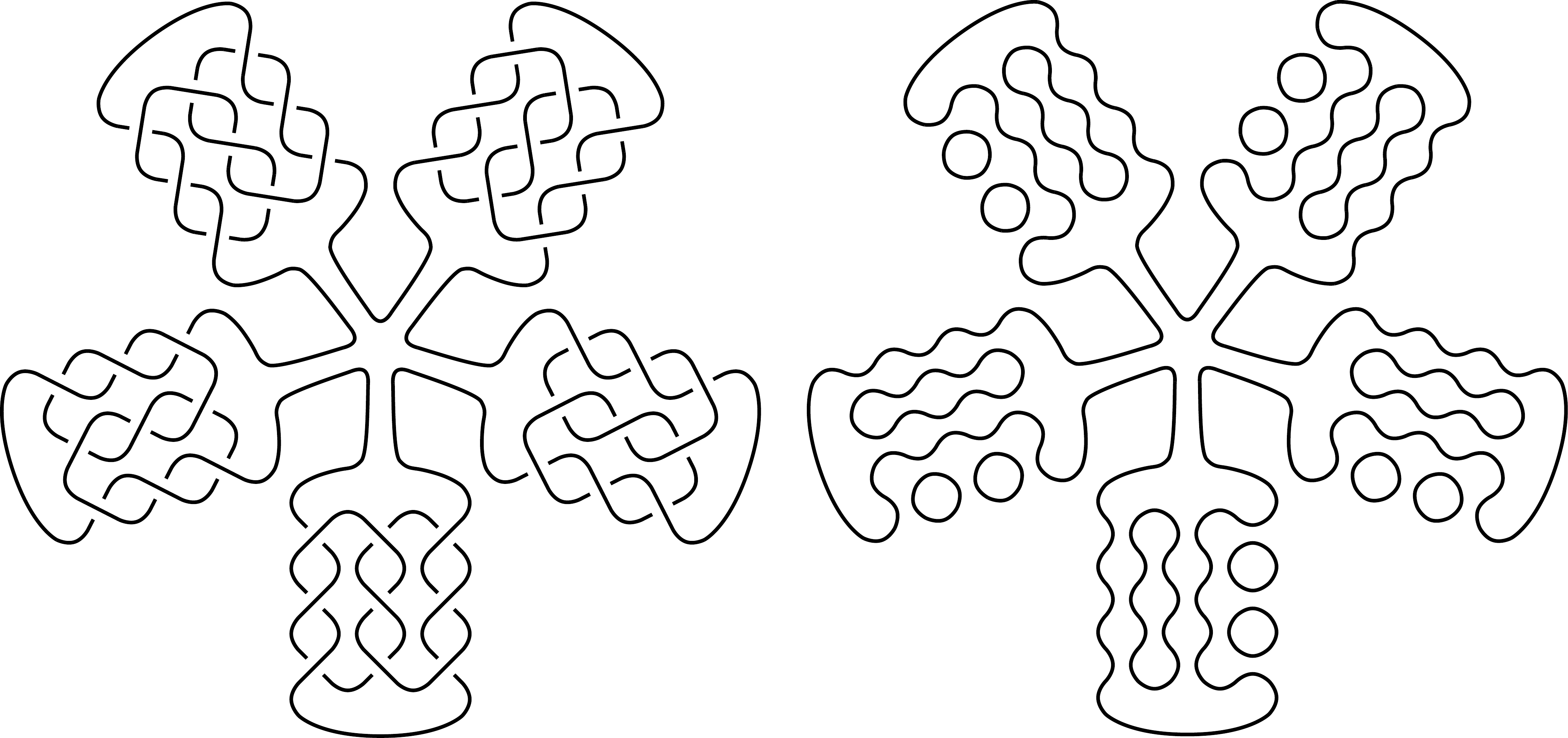}
		\caption{The windmill diagram (left) for $\#_5 (9_{46})$ and an enhanced state windmill (right) from the relative Khovanov-Jacobsson cycle associated to one of its slices.}
		\label{fig_windmill}
	\end{figure}
	
	\begin{proof}
		Use a \textit{windmill} diagram such as the one in Figure \ref{fig_windmill} \!, with $k$-many inward-facing copies of $9_{46}$ as the \textit{sails} of the windmill and $k$-many \textit{windshafts} which connect the sails in the center of the diagram. Movies for the $2^k$-many slices are obtained by choosing either of the movies from Theorem \ref{obpic_theorem3} for each sail and using $k$-many Morse saddles to merge their final frames along the windshafts. 
	
		By connecting the diagrams in this way, the relative Khovanov-Jacobsson cycle associated to one of these movies consists of $8^k$-many \textit{enhanced state windmills} whose sails are one of the $8$-many enhanced states from $\ckj_{\Sigma_\ell}$ or $\ckj_{\Sigma_r}$, depending on which movie was chosen for the corresponding sail in the diagram. The enhanced state windmills have this form because the final $k$-many Morse saddles forming the windshafts will connect the enhanced states from $\ckj_{\Sigma_\ell}$ and $\ckj_{\Sigma_r}$ along the ``top loop" in the smoothing; because all of these loops are $1$-labeled, the Morse saddles will preserve this labeling. In the aforementioned figure, we have illustrated an enhanced state windmill for the slice of $\#_k (9_{46})$ having movies of $\Sigma_r$ on all but one sail. 
		
		A more general way to describe these slices is as the boundary-connected sum of $k$-many 
		copies of $\Sigma_\ell$ or $\Sigma_r$. The utility of the (longwinded) process above is a specified movie description.
		
		We now prove the result. Choose a pair of slices $\Sigma_{0,1}$ from the $2^k$-many slices described above, and produce their relative Khovanov-Jacobsson cycles per the given movie descriptions. As in the proof of Theorem \ref{obpic_theorem3} \!, we will trim the diagrams to simplify the cycles, although in this case, we must trim each sail. Trim the left column of a sail when the movie of $\Sigma_0$ shows $\Sigma_r$ and trim the right column when it shows $\Sigma_\ell$. Because $\Sigma_0$ and $\Sigma_1$ differ in their choice of movies for the sails, one of the trims will kill $\ckj_{\Sigma_1}$. The only remaining element of $\ckj_{\Sigma_0}$ will be the all $1$-label of the $1$-state of the trimmed diagram of $\#_k(9_{46})$. For example, we have shown the trim of the enhanced state windmill from Figure 11 here. As before, such an element represents a nontrivial homology class in the Khovanov homology of this trimmed diagram. The result follows just as it did in the previous two results.
	\end{proof}
	
	\begin{corollary}
		The knot $\#_k \big(P(n,-n,n)\big)$ has at least $2^k$-many distinct slices.
	\end{corollary}
	\begin{proof}
		Similar to Corollary \ref{obpic_corollary4} \!, the number of crossing used in the strands for each sail of the windmill does not affect the argument in Corollary \ref{obpic_corollary5} \!.
	\end{proof}
	
\end{section}

%
%


\bibliographystyle{alpha}
\bibliography{mybib}

\begin{thebibliography}{Kho06}

\bibitem[BN02]{barnatan02}
Dror Bar-Natan.
\newblock On {K}hovanov's categorification of the {J}ones polynomial.
\newblock {\em Algebr. Geom. Topol.}, \textbf{2}:337--370, 2002.
\newblock MR 1917056 Zbl 0998.57016.

\bibitem[BN05]{barnatan05}
Dror Bar-Natan.
\newblock {K}hovanov's homology for tangles and cobordisms.
\newblock {\em Geom. Topol.}, \textbf{9}:1443--1499, 2005.
\newblock MR2174270 Zbl 1084.57011.

\bibitem[CP20]{conwaypowell20}
Anthony Conway and Mark Powell.
\newblock Characterisation of homotopy ribbon disks.
\newblock {\em arXiv:1902.05321}, 2020.

\bibitem[CSS04]{cartersaitosatoh04}
J.~Scott Carter, Masahico Saito, and Shin Satoh.
\newblock Ribbon-moves for 2-knots with 1-handles attached and
  {K}hovanov-{J}acobsson numbers.
\newblock \textbf{134}(9):2779--2783, 2004.
\newblock MR2213759.

\bibitem[Ell09]{elliott09_1}
Andrew Elliott.
\newblock Graphical methods establishing nontriviality of state cycle
  {K}hovanov homology classes.
\newblock Unpublished note, arXiv:0907.0396, 2009.

\bibitem[Ell10]{elliott09_2}
Andrew Elliott.
\newblock {\em State cycles, quasipositive modification, and constructing
  {H}-thick knots in {K}hovanov homology}.
\newblock PhD thesis, Rice University, 2010.
\newblock \textit{Proquest LLC, Ann Arbor, MI}, 2010. 108 pp. MR2782376.

\bibitem[GJ11]{greenejabuka11}
Joshua Greene and Stanislav Jabuka.
\newblock The slice-ribbon conjecture for 3-stranded pretzel knots.
\newblock {\em Amer. J. Math.}, \textbf{133}(3):555--580, 2011.
\newblock MR2808326.

\bibitem[GL20]{gujrallevine20}
Onkar~Singh Gujral and Adam~Simon Levine.
\newblock {K}hovanov homology and cobordisms between split links.
\newblock {\em arXiv:2009.03406}, 2020.

\bibitem[Jac04]{jacobsson04}
Magnus Jacobsson.
\newblock An invariant of link cobordisms from {K}hovanov homology.
\newblock {\em Algebr. Geom. Topol.}, \textbf{4}:1211--1251, 2004.
\newblock MR2113903.

\bibitem[JM16]{juhaszmarengon16}
Andr\'as Juh\'asz and Marco Marengon.
\newblock Concordance maps in knot {F}loer homology.
\newblock {\em Geom. Topol.}, \textbf{20}(6):3623--3673, 2016.
\newblock MR3590358.

\bibitem[JZ20]{juhaszzemke20}
Andr\'as Juh\'asz and Ian Zemke.
\newblock Distinguishing slice disks using knot floer homology.
\newblock {\em Seceta Math. (N.S.)}, \textbf{20}(1), 2020.
\newblock Paper No. 5, 18 pp. MR4045151.

\bibitem[Kho00]{khovanov00}
Mikhail Khovanov.
\newblock A categorification of the {J}ones polynomial.
\newblock {\em Duke Math. J.}, \textbf{101}(3):359--426, 2000.
\newblock MR1740682.

\bibitem[Kho06]{khovanov06}
Mikhail Khovanov.
\newblock An invariant of tangle cobordisms.
\newblock {\em Trans. Amer. Math. Soc.}, \textbf{358}(1):315--327, 2006.
\newblock MR2171235.

\bibitem[MP20]{millerpowell20}
Allison~N. Miller and Mark Powell.
\newblock Stabilization distance between surfaces.
\newblock {\em Enseign. Math.}, \textbf{65}(3-4):397--440, 2020.
\newblock MR4113047.

\bibitem[Ras05]{rasmussen05}
Jacob Rasmussen.
\newblock {K}hovanov's invariant for closed surfaces.
\newblock Unpublished note, arXiv:math/0502527, 2005.

\bibitem[Swa10]{swann10}
Jonah Swann.
\newblock {\em Relative {K}hovanov-{J}acobsson {C}lasses}.
\newblock PhD thesis, Bryn Mawr College, 2010.

\bibitem[Tan05]{tanaka05}
Kokoro Tanaka.
\newblock {K}hovanov-{J}acobsson numbers and invariants of surface-knots
  derived from {B}ar-{N}atan's theory.
\newblock {\em Proc. Amer. Math. Soc.}, \textbf{134}(12):3685--3689, 2005.
\newblock MR2240683.

\end{thebibliography}

\end{document}